\documentclass[oneside,reqno,english]{amsart}
\usepackage{lmodern}
\usepackage[T1]{fontenc}
\usepackage[latin9]{inputenc}
\usepackage{babel}
\usepackage{varioref}
\usepackage{float}
\usepackage{amstext}
\usepackage{amsthm}
\usepackage{amssymb}
\usepackage[unicode=true,pdfusetitle,
 bookmarks=true,bookmarksnumbered=false,bookmarksopen=false,
 breaklinks=false,pdfborder={0 0 1},backref=false,colorlinks=false]
 {hyperref}

\makeatletter

\floatstyle{ruled}
\newfloat{algorithm}{tbp}{loa}
\providecommand{\algorithmname}{Algorithm}
\floatname{algorithm}{\protect\algorithmname}

\numberwithin{equation}{section}
\numberwithin{figure}{section}
\newenvironment{lyxcode}
	{\par\begin{list}{}{
		\setlength{\rightmargin}{\leftmargin}
		\setlength{\listparindent}{0pt}
		\raggedright
		\setlength{\itemsep}{0pt}
		\setlength{\parsep}{0pt}
		\normalfont\ttfamily}%
	 \item[]}
	{\end{list}}
\theoremstyle{plain}
\newtheorem{thm}{\protect\theoremname}[section]
\theoremstyle{plain}
\newtheorem{prop}[thm]{\protect\propositionname}
\theoremstyle{plain}
\newtheorem{assumption}[thm]{\protect\assumptionname}
\theoremstyle{remark}
\newtheorem{rem}[thm]{\protect\remarkname}
\theoremstyle{plain}
\newtheorem{lyxalgorithm}[thm]{\protect\algorithmname}
\theoremstyle{remark}
\newtheorem*{claim*}{\protect\claimname}

\newcommand{\sspan}{\mbox{\rm span}}

\makeatother

\providecommand{\algorithmname}{Algorithm}
\providecommand{\assumptionname}{Assumption}
\providecommand{\claimname}{Claim}
\providecommand{\propositionname}{Proposition}
\providecommand{\remarkname}{Remark}
\providecommand{\theoremname}{Theorem}

\begin{document}
\begin{lyxcode}
\title[Dykstra linear convergence: sets case]{Linear convergence of distributed Dykstra's algorithm for sets under an intersection property} 
\end{lyxcode}

\subjclass[2010]{68Q25, 68W15, 90C25, 90C30, 65K05}
\begin{abstract}
We show the linear convergence of a distributed Dykstra's algorithm
for sets intersecting in a manner slightly stronger than the usual
constraint qualifications.
\end{abstract}

\author{C.H. Jeffrey Pang}
\thanks{C.H.J. Pang acknowledges grant R-146-000-265-114 from the Faculty
of Science, National University of Singapore. }
\curraddr{Department of Mathematics\\ 
National University of Singapore\\ 
Block S17 08-11\\ 
10 Lower Kent Ridge Road\\ 
Singapore 119076 }
\email{matpchj@nus.edu.sg}
\date{\today{}}
\keywords{Distributed optimization, Dykstra's algorithm, linear convergence}

\maketitle
\tableofcontents{}

\section{Introduction}

Let $G=(V,E)$ be an undirected graph. For all $i\in V$, let $C_{i}\subset\mathbb{R}^{m}$
be closed convex sets, and $\bar{x}_{i}\in\mathbb{R}^{m}$. For a
closed convex set $C$, let $\delta_{C}(\cdot)$ be its indicator
function. Consider the distributed optimization problem
\begin{equation}
\begin{array}{c}
\underset{x\in\mathbb{R}^{m}}{\min}\underset{i\in V}{\overset{\phantom{i\in V}}{\sum}}\left[\delta_{C_{i}}(x)+\frac{1}{2}\|x-\bar{x}_{i}\|^{2}\right],\end{array}\label{eq:dist_opt_pblm}
\end{equation}
where communications between two vertices in $V$ occur only along
edges in $E$. In Remark \ref{rem:WLOG-bar-x-equal}, we explain that
we can assume that all $\bar{x}_{i}$ are equal to some $\bar{x}$
without losing any generality. The problem is therefore equivalent
to projecting $\bar{x}$ onto $\cap_{i\in V}C_{i}$ in a distributed
manner.

\subsection{A review of the distributed Dykstra's splitting}

In our earlier paper \cite{Pang_Dist_Dyk}, we considered the more
general problem than \eqref{eq:dist_opt_pblm} where $\delta_{C_{i}}(\cdot)$
can be general closed convex functions instead. We proposed a deterministic
distributed asynchronous decentralized algorithm based on dual ascent
for \eqref{eq:dist_opt_pblm} that converges to the primal minimizer,
and call it the distributed Dykstra's algorithm. Our approach was
motivated by work on Dykstra's algorithm in \cite{Dykstra83,BD86,Gaffke_Mathar,Hundal-Deutsch-97}.
See also \cite{Han88}. We also remark that the dual ascent idea had
been discussed in \cite{Combettes_Dung_Vu_JMAA,Combettes_Dung_Vu_SVAA,Pesquet_Abboud_gang_2017_JMIV}.
We refer to the introduction in \cite{Pang_Dist_Dyk} for more historical
summary of these methods. Part of the contribution in \cite{Pang_Dist_Dyk}
was to point out that the dual ascent idea leads to a desirable distributed
optimization algorithm. We give more details of the distributed Dykstra's
algorithm in Section \ref{sec:Preliminaries}.

\subsection{Linear convergence of Dykstra's algorithm}

A well known algorithm for solving \eqref{eq:dist_opt_pblm} is Dykstra's
algorithm. The primal problem and its corresponding (Fenchel) dual
are typically written as 
\[
\begin{array}{c}
\underset{x\in\mathbb{R}^{m}}{\min}\frac{1}{2}\|x-\bar{x}\|^{2}+\underset{i\in V}{\overset{}{\sum}}\delta_{C_{i}}(x)\text{ and }\underset{z_{i}\in\mathbb{R}^{m},i\in V}{\max}\frac{1}{2}\|\bar{x}\|^{2}-\frac{1}{2}\left\Vert \bar{x}-\underset{i\in V}{\overset{\phantom{i\in V}}{\sum}}z_{i}\right\Vert ^{2}-\underset{i\in V}{\overset{\phantom{i\in V}}{\sum}}\delta_{C_{i}}^{*}(z_{i})\end{array}
\]
respectively, and solved by block coordinate maximization on the dual
problem. (See \cite{BD86,Han88,Gaffke_Mathar}). (Note that this dual
is different from \eqref{eq:dist-dyk-Fenchel-dual}.) In the case
when $C_{i}$ are halfspaces, linear convergence of Dykstra's algorithm
was established in \cite{Iusem_DePierro_Hildreth}, with refined rates
given in \cite{Deutsch_Hundal_rate_Dykstra}. We extended the linear
rates to polyhedra in \cite{Pang_lin_rate_Dykstra}.

A linear convergence rate of Dykstra's algorithm assures that a high
accuracy solution can be obtained in a reasonable amount of time.
This would then allow the algorithm to be used as a subroutine of
other optimization algorithms. For example, the distributed optimization
algorithms \cite{Aybat_Hamedani_2016,Scutari_ASY_SONATA_2018} (and
perhaps many others) make use of the averaged consensus algorithm
as a subroutine. (The linear convergence rate of averaged consensus
is used in the convergence proof of the main distributed optimization
algorithm.) Since averaged consensus is a particular case of the distributed
Dykstra's algorithm with all $C_{i}$ being $\mathbb{R}^{m}$, it
is plausible to make use of the distributed Dykstra's algorithm to
help solve constrained distributed problems. 

\subsection{Contributions of this paper}

Even though we have observed linear convergence rates of the distributed
Dykstra's algorithm in \cite{Pang_rate_D_Dyk} in our numerical experiments
for the case when some of the terms are indicator functions of closed
convex sets, it seems that there is no theoretical justification yet
of linear convergence for both Dykstra's original algorithm and for
the distributed Dykstra's algorithm beyond the polyhedral case. As
is well-known, the intersection $\cap_{i\in V}C_{i}$ can be sensitive
to the perturbation of the sets $C_{i}$ \cite{Kruger_06}, so additional
constraint qualifications are needed for the linear convergence of
the method of alternating projections (see for example \cite{BB96_survey}).

In this paper, we prove the asymptotic linear convergence of the distributed
Dykstra's algorithm when the functions are indicator functions of
sets that are not necessarily polyhedral. We assume that the sets
satisfy a property on systems of intersections of sets stronger than
what is typically studied in the method of alternating projections.
We also make assumptions that are closely related to conditions used
to prove linear convergence in proximal algorithms. 

\subsection{Notation}

Variables in bold, like $\mathbf{x}$ and $\mathbf{z}_{i}$, typically
lie in the space $[\mathbb{R}^{m}]^{|V|}$, while variables not in
bold, like $x$ and $y$, typically lie in $\mathbb{R}^{m}$. All
norms shall be the 2-norm. We often use ``$\hat{\phantom{x}}"$ to
represent the unit vector in a given direction. For example, $\hat{x}_{i}^{0}=\frac{x_{i}^{0}}{\|x_{i}^{0}\|}$. 

\section{Preliminaries \label{sec:Preliminaries}}

In this section, we lay down the preliminaries of the paper. 

For each $i\in V$, let $\mathbf{f}_{i}:[\mathbb{R}^{m}]^{|V|}\to\mathbb{R}\cup\{\infty\}$
be defined by 
\begin{equation}
\mathbf{f}_{i}(\mathbf{x})=\delta_{C_{i}}([\mathbf{x}]_{i}).\label{eq:f-component}
\end{equation}
For each $(i,j)\in E$, define the halfspaces $H_{(i,j)}$ to be 
\[
H_{(i,j)}:=\{\mathbf{x}\in[\mathbb{R}^{m}]^{|V|}:\mathbf{x}_{i}=\mathbf{x}_{j}\}.
\]
Since the graph is connected, the intersection of all these halfspaces
is the diagonal set defined by 
\begin{equation}
D:=\cap_{e\in E}H_{e}=\{\mathbf{x}\in[\mathbb{R}^{m}]^{|V|}:\mathbf{x}_{1}=\mathbf{x}_{2}=\cdots=\mathbf{x}_{|V|}\}.\label{eq:def-D-hyperplane}
\end{equation}
For each $e\in E$, define $\mathbf{f}_{e}:[\mathbb{R}^{m}]^{|V|}\to\mathbb{R}$
by $\mathbf{f}_{e}(\mathbf{x})=\delta_{H_{e}}(\mathbf{x})$. The setting
for the distributed Dykstra's algorithm that is easily seen to be
equivalent to \eqref{eq:dist_opt_pblm} is 
\begin{equation}
\begin{array}{c}
\underset{\mathbf{x}\in[\mathbb{R}^{m}]^{|V|}}{\min}\frac{1}{2}\|\mathbf{x}-\bar{\mathbf{x}}\|^{2}+\underset{i\in V}{\sum}\mathbf{f}_{i}(\mathbf{x})+\underset{e\in E}{\overset{\phantom{\alpha\in\bar{E}}}{\sum}}\delta_{H_{e}}(\mathbf{x}),\end{array}\label{eq:general-framework}
\end{equation}
where $\bar{\mathbf{x}}\in[\mathbb{R}^{m}]^{|V|}$ is such that each
component of $[\bar{\mathbf{x}}]_{i}$, where $i\in V$, is equal
to $\bar{x}$. Let the dual variables be $\mathbf{z}=\{\mathbf{z}_{\alpha}\}_{\alpha\in V\cup E}$,
where each $\mathbf{z}_{\alpha}\in[\mathbb{R}^{m}]^{|V|}$. The (Fenchel)
dual of \eqref{eq:general-framework} can be calculated to be 
\begin{equation}
\begin{array}{c}
\underset{\mathbf{z}_{\alpha}\in[\mathbb{R}^{m}]^{|V|},\alpha\in V\cup E}{\max}\frac{1}{2}\|\bar{\mathbf{x}}\|^{2}-\frac{1}{2}\bigg\|\bar{\mathbf{x}}-\underset{\alpha\in V\cup E}{\overset{\phantom{i\in V}}{\sum}}\mathbf{z}_{\alpha}\bigg\|^{2}-\underset{i\in V}{\overset{\phantom{i\in V}}{\sum}}\delta_{C_{i}}^{*}(z_{i})-\underset{e\in E}{\overset{\phantom{i\in V}}{\sum}}\delta_{H_{e}}^{*}(\mathbf{z}_{e}).\end{array}\label{eq:dist-dyk-Fenchel-dual}
\end{equation}

\begin{prop}
\label{prop:sparsity}(Sparsity) If the value in \eqref{eq:dist-dyk-Fenchel-dual}
is finite, then 
\begin{enumerate}
\item If $i\in V$, then $\mathbf{z}_{i}\in[\mathbb{R}^{m}]^{|V|}$ is such
that $[\mathbf{z}_{i}]_{j}=0$ for all $j\in V\backslash\{i\}$.
\item If $(i,j)\in E$, then $\mathbf{z}_{(i,j)}\in[\mathbb{R}^{m}]^{|V|}$
is such that $[\mathbf{z}_{(i,j)}]_{k}=0$ for all $k\in V\backslash\{i,j\}$,
and $[\mathbf{z}_{(i,j)}]_{i}+[\mathbf{z}_{(i,j)}]_{j}=0$. 
\end{enumerate}
\end{prop}

\begin{proof}
The proof is elementary and exactly the same as that in \cite{Pang_Dist_Dyk}.
(Part (1) makes use of the fact that $\mathbf{f}_{i}(\cdot)$ depends
on only the $i$-th coordinate of the input, while part (2) makes
use of the fact that $\delta_{H_{(i,j)}}^{*}(\cdot)=\delta_{H_{(i,j)}^{\perp}}(\cdot)$,
and $\delta_{H_{(i,j)}^{\perp}}(\mathbf{z}_{(i,j)})<\infty$ if and
only if the conditions in (2) hold.)
\end{proof}
In view of Proposition \ref{prop:sparsity}, the vector $\mathbf{z}_{i}$
for all $i\in V$ are such that $[\mathbf{z}_{i}]_{j}=0$ if $j\neq i$.
Letting $z_{i}:=[\mathbf{z}_{i}]_{i}$, we let the dual function $F:[[\mathbb{R}^{m}]^{|V|}]^{|V\cup E|}\to\mathbb{R}$
be 
\begin{align}
F(\mathbf{z}) & :=\sum_{i\in V}\delta_{C_{i}}^{*}(z_{i})+\sum_{e\in E}\delta_{H_{e}}^{*}(\mathbf{z}_{e})+\frac{1}{2}\bigg\|\underbrace{\bar{\mathbf{x}}-\sum_{\alpha\in V\cup E}\mathbf{z}_{\alpha}}_{=:\mathbf{x}}\bigg\|^{2}.\label{eq:def-dual-F}
\end{align}
It is clear to see that $F(\mathbf{z})$ differs from \eqref{eq:dist-dyk-Fenchel-dual}
by a sign and a constant. It is known that strong duality between
\eqref{eq:general-framework} and \eqref{eq:dist-dyk-Fenchel-dual}
holds (even though a dual minimizer may not exist). Minimizing $F(\cdot)$
allows one to find the optimal value to \eqref{eq:dist-dyk-Fenchel-dual},
and also the optimal solution to \eqref{eq:general-framework}. It
turns out that the only variables that need to be tracked are $z_{i}\in\mathbb{R}^{m}$
for all $i\in V$ and $\mathbf{x}\in[\mathbb{R}^{m}]^{|V|}$ as marked
above. We shall prove that $\mathbf{x}$ converges linearly to the
optimal primal solution under some additional assumptions. We refer
to the $i$-th coordinate of $\mathbf{x}$ as $x_{i}$. Also, if $x^{*}$,
the projection of $\bar{x}$ onto $\cap_{i\in V}C_{i}$, were to be
zero, then $F(\mathbf{z})$ takes the minimum of zero when $\mathbf{x}$
is the primal optimal solution and $\{z_{i}\}_{i\in V}$ are optimal
multipliers. 

Here are the first set of assumptions we need to prove our linear
convergence result. 
\begin{assumption}
\label{assu:the-assu}Suppose that the following assumptions hold:
\begin{enumerate}
\item Let $x^{*}\in\mathbb{R}^{m}$ be the optimal solution to \eqref{eq:dist_opt_pblm}.
We assume that $x^{*}=0$.
\item \label{enu:all-bar-x-equal}The $\bar{x}_{i}$ are all equal for all
$i\in V$. 
\item (Existence of dual minimizers) There exists $\{z_{i}\}_{i\in V}$
such that $z_{i}\in N_{C_{i}}(x^{*})$ and $\sum_{i\in V}z_{i}=|V|\bar{x}$. 
\item \label{enu:lin-reg-C-i}(Regularity of the sets $C_{i}$) The sets
satisfy a nondegeneracy constraint qualification: There is a neighborhood
$U$ of $x^{*}$ and parameters $M_{\max}>1$ and $M_{\min}>0$ such
that if the multipliers $\{z_{i}\}_{i\in V}$ and points $\{x_{i}\}_{i\in V}$
are such that $x_{i}\in U$ and $z_{i}\in N_{C_{i}}(x_{i})$ for all
$i\in V$ and $[\bar{\mathbf{x}}-\sum_{\alpha\in V\cup E}\mathbf{z}_{\alpha}]_{j}\in U$
for all $j\in V$, then 
\begin{equation}
M_{\min}\leq\|z_{i}\|\leq M_{\max}\text{ for all }i\in V.\label{eq:norm-z-i-bdd}
\end{equation}
Let $H_{i}$ be the hyperplane $\{x:z_{i}^{T}(x-x_{i})=0\}$ for
all $i\in V$. Assume that for all $x\in\mathbb{R}^{m}$, there is
some constant $\kappa_{1}>0$ such that $d(x,\cap_{i\in V}H_{i})\leq\kappa_{1}\max_{i\in V}d(x,H_{i})$. 
\item (Graph connectedness) The (undirected) graph $G=(V,E)$ is connected.
\item (Semismoothness) The sets satisfy the semismoothness property of order
2 at $x^{*}$: For a point $x_{i}\in\partial C_{i}$ near $x^{*}$,
let a supporting hyperplane to $x_{i}$ at $C_{i}$ with normal $z_{i}\in N_{C_{i}}(x_{i})$
be $H_{i}$. Then we have $d(x^{*},H_{i})=O(\|x-x^{*}\|^{2})$. {[}We
know that all convex sets satisfy the property if $O(\|x-x^{*}\|^{2})$
were replaced by $o(\|x-x^{*}\|)$.{]} Suppose $z_{i}\in N_{C_{i}}(x_{i})$.
Since $d(x^{*},H_{i})=\hat{z}_{i}^{T}x_{i}$, there is a $\kappa_{2}>0$
such that
\begin{equation}
\begin{array}{c}
\delta_{C_{i}}^{*}(z_{i})=\langle z_{i},x_{i}\rangle=\|z_{i}\|\langle\frac{z_{i}}{\|z_{i}\|},x_{i}\rangle\leq\kappa_{2}\|z_{i}\|\|x_{i}\|^{2}._{\phantom{C_{i}}}^{\phantom{C_{i}}}\end{array}\label{eq:delta-star-has-squared-values}
\end{equation}
\item \label{enu:osc-of-norms-hat}(First order property on normals) There
is a neighborhood $U$ of $x^{*}$ and $\kappa_{3}>0$ such that for
all $i\in V$, if $x\in U\cap C_{i}$, and $z\in N_{C_{i}}(x)\backslash\{0\}$,
then there is a $z^{r}\in N_{C_{i}}(x^{*})\backslash\{0\}$ such that
\begin{equation}
\begin{array}{c}
\left\Vert \frac{z}{\|z\|}-\frac{z^{r}}{\|z^{r}\|}\right\Vert \leq\kappa_{3}\|x-x^{*}\|.\end{array}\label{eq:normals-linear-dec}
\end{equation}
\item \label{enu:lin-reg}(A linear regularity property on the normal cones)
Define the set $M\subset[\mathbb{R}^{m}]^{|V|}$ of optimal multipliers
to be $M:=M_{1}\cap M_{2}$, where \begin{subequations}\label{eq_m:set-M}
\begin{align}
M_{1}:= & \begin{array}{c}
N_{C_{1}}(x^{*})\times\cdots\times N_{C_{|V|}}^{\phantom{C_{|V|}}}(x^{*})\end{array}\label{eq:set-M1}\\
\text{and }M_{2}:= & \begin{array}{c}
\Big\{\mathbf{z}\in[\mathbb{R}^{m}]^{|V|}:\underset{i\in V}{\overset{\phantom{i\in V}}{\sum}}\mathbf{z}_{i}=|V|\bar{x}\Big\}.\end{array}\label{eq:set-M2}
\end{align}
\end{subequations}Assume there is a $\kappa_{4}>0$ such that 
\begin{equation}
d(\mathbf{z},M_{1}\cap M_{2})\leq\kappa_{4}d(\mathbf{z},M_{2})\text{ for all }\mathbf{z}\in M_{1}.\label{eq:lin-reg-ineq}
\end{equation}
\end{enumerate}
\end{assumption}

We remark about Assumption \ref{assu:the-assu}\eqref{enu:lin-reg}.
The linear regularity property is usually stated as $d(\mathbf{z},M_{1}\cap M_{2})\leq\kappa_{4}\max\{d(\mathbf{z},M_{1}),d(\mathbf{z},M_{2})\}$
for all $\mathbf{z}$, but we state a weaker version of it in Assumption
\ref{assu:the-assu}\eqref{enu:lin-reg} because that is what our
proof needs. The stronger linear regularity is satisfied whenever
the normal cones $N_{C_{i}}(x^{*})$ are polyhedral (see for example
\cite[Corollary 5.26]{BB96_survey}), so this assumption is quite
reasonable.

Assumption \ref{assu:the-assu}\eqref{enu:lin-reg-C-i} is stronger
than the usual transversality condition typically studied in the method
of alternating projections. Now that we are working with an optimization
problem \eqref{eq:dist_opt_pblm} rather than a feasibility problem,
it may be more appropriate to compare to the Robinson constraint qualification.
We seek to study this assumption further in future work.

We make the following remark.
\begin{rem}
\label{rem:WLOG-bar-x-equal}(On Assumption \ref{assu:the-assu}\eqref{enu:all-bar-x-equal})
We now show that Assumption \ref{assu:the-assu}\eqref{enu:all-bar-x-equal}
does not lose any generality. Suppose that the $\bar{x}_{i}$ are
not all necessarily the same. Note that $\sum_{i\in V}\frac{1}{2}\|x-\bar{x}_{i}\|^{2}=\sum_{i\in V}(\frac{1}{2}\|x-a\|^{2}+\frac{1}{2}\|\bar{x}_{i}\|^{2}-\frac{1}{2}\|a\|^{2})$,
where $a=\frac{1}{|V|}\sum_{i\in V}\bar{x}_{i}$. Thus all the $\bar{x}_{i}$
can be replaced by $a$. Note that this does not mean that the primal
iterate $\mathbf{x}$ needs to be such that all its coordinates are
$a$ at the start. 
\end{rem}

We now state Algorithm \vref{alg:Ext-Dyk}, which minimizes $F(\cdot)$
by block coordinate minimization.

\begin{algorithm}[!h]
\begin{lyxalgorithm}
\label{alg:Ext-Dyk}(Distributed Dykstra's algorithm) Our distributed
Dykstra's algorithm is as follows:

01$\quad$Let 

\begin{itemize}
\item $\mathbf{z}_{i}^{1,0}\in[\mathbb{R}^{m}]^{|V|}$ be a starting dual
vector for $\mathbf{f}_{i}(\cdot)$ for each $i\in V$ so that $[\mathbf{z}_{i}^{1,0}]_{j}=0$
for all $j\in V\backslash\{i\}$. 
\item $\mathbf{z}_{(i,j)}^{1,0}\in[\mathbb{R}^{m}]^{|V|}$ be a starting
dual vector for each edge $(i,j)$ so that $[\mathbf{z}_{(i,j)}]_{i}+[\mathbf{z}_{(i,j)}]_{j}=0$
and $[\mathbf{z}_{(i,j)}]_{i'}=0$ for all $i'\in V\backslash\{i,j\}$.
\end{itemize}
02$\quad$For $n=1,2,\dots$

03$\quad$$\quad$For $w=1,2,\dots,\bar{w}$

04$\quad$$\quad$$\quad$Choose a set $S_{n,w}\subset E\cup V$ such
that $S_{n,w}\neq\emptyset$. 

05$\quad$$\quad$$\quad$Define $\{z_{\alpha}^{n,w}\}_{\alpha\in S_{n,w}}$
by 
\begin{equation}
\{\mathbf{z}_{\alpha}^{n,w}\}_{\alpha\in S_{n,w}}=\underset{\mathbf{z}_{\alpha},\alpha\in S_{n,w}}{\arg\min}\frac{1}{2}\left\Vert \bar{\mathbf{x}}-\sum_{\alpha\notin S_{n,w}}\mathbf{z}_{\alpha}^{n,w-1}-\sum_{\alpha\in S_{n,w}}\mathbf{z}_{\alpha}\right\Vert ^{2}+\sum_{\alpha\in S_{n,w}}\mathbf{f}_{\alpha}^{*}(\mathbf{z}_{\alpha}).\label{eq:Dykstra-min-subpblm}
\end{equation}

06$\quad$$\quad$$\quad$Set $\mathbf{z}_{\alpha}^{n,w}:=\mathbf{z}_{\alpha}^{n,w-1}$
for all $\alpha\notin S_{n,w}$.

07$\quad$$\quad$End For 

08$\quad$$\quad$Let $\mathbf{z}_{\alpha}^{n+1,0}=\mathbf{z}_{\alpha}^{n,\bar{w}}$
for all $\alpha\in V\cup E$.

09$\quad$End For 
\end{lyxalgorithm}

\end{algorithm}

To provide some intuition to Algorithm \ref{alg:Ext-Dyk}, we mention
that minimizing only one $\mathbf{z}_{i}$ at a time for some $i\in V$
(i.e., $S_{n,k}=\{i\}$) reduces \eqref{eq:Dykstra-min-subpblm} to
a standard proximal problem. Minimizing only one $\mathbf{z}_{(i,j)}$
for some $(i,j)\in E$ (i.e., $S_{n,k}=\{(i,j)\}$) has the natural
interpretation of averaging the $i$-th and $j$-th components of
$\mathbf{x}$.

Let the function $\mathbf{f}_{e}:[\mathbb{R}^{m}]^{|V|}\to\mathbb{R}\cup\{\infty\}$
to be defined to be $\mathbf{f}_{e}(\cdot)=\delta_{H_{e}}(\cdot)$.
Let $\mathbf{x}^{*}$ be the optimal solution of \eqref{eq:general-framework}.
Before we prove the result, we note that using a technique in \cite{Gaffke_Mathar},
the duality gap between the primal and dual pair \eqref{eq:general-framework}
and \eqref{eq:dist-dyk-Fenchel-dual} satisfies 
\begin{eqnarray}
 &  & \begin{array}{c}
\frac{1}{2}\|\mathbf{x}^{*}-\bar{\mathbf{x}}\|^{2}+\underset{\alpha\in V\cup E}{\overset{\phantom{\alpha\in E}}{\sum}}\mathbf{f}_{\alpha}(\mathbf{x}^{*})-\frac{1}{2}\|\bar{\mathbf{x}}\|^{2}+\frac{1}{2}\bigg\|\bar{\mathbf{x}}-\underset{\alpha\in V\cup E}{\overset{\phantom{\alpha\in E}}{\sum}}\mathbf{z}_{\alpha}\bigg\|^{2}+\underset{\alpha\in V\cup E}{\overset{\phantom{\alpha\in E}}{\sum}}\mathbf{f}_{\alpha}^{*}(\mathbf{z}_{\alpha})\end{array}\nonumber \\
 & = & \begin{array}{c}
\frac{1}{2}\|\mathbf{x}^{*}-\bar{\mathbf{x}}\|^{2}+\underset{\alpha\in E\cup V}{\sum}[\mathbf{f}_{\alpha}(\mathbf{x}^{*})+\mathbf{f}_{\alpha}^{*}(\mathbf{z}_{\alpha})]-\left\langle \bar{\mathbf{x}},\underset{\alpha\in E\cup V}{\sum}\mathbf{z}_{\alpha}\right\rangle +\frac{1}{2}\left\Vert \underset{\alpha\in E\cup V}{\sum}\mathbf{z}_{\alpha}\right\Vert ^{2}\end{array}\nonumber \\
 & \overset{\scriptsize\mbox{Fenchel duality}}{\geq} & \begin{array}{c}
\frac{1}{2}\|\mathbf{x}^{*}-\bar{\mathbf{x}}\|^{2}+\left\langle \mathbf{x}^{*},\underset{\alpha\in E\cup V}{\sum}\mathbf{z}_{\alpha}\right\rangle -\left\langle \bar{\mathbf{x}},\underset{\alpha\in E\cup V}{\sum}\mathbf{z}_{\alpha}\right\rangle +\frac{1}{2}\left\Vert \underset{\alpha\in E\cup V}{\sum}\mathbf{z}_{\alpha}\right\Vert ^{2}\end{array}\nonumber \\
 & = & \begin{array}{c}
\frac{1}{2}\left\Vert \mathbf{x}^{*}-\bar{\mathbf{x}}+\underset{\alpha\in E\cup V}{\sum}\mathbf{z}_{\alpha}\right\Vert ^{2}\overset{\eqref{eq:def-dual-F}}{=}\frac{1}{2}\|\mathbf{x}^{*}-\mathbf{x}\|^{2}.\end{array}\label{eq:fenchel-duality-gap}
\end{eqnarray}
The strategy behind our linear convergence proof is to show that the
duality gap in the first line of \eqref{eq:fenchel-duality-gap} converges
linearly to zero, which will force the last formula of \eqref{eq:fenchel-duality-gap}
to converge linearly to zero, which in turn shows the linear convergence
of $\mathbf{x}$ to $\mathbf{x}^{*}$. Note that since $\mathbf{x}^{*}=0$,
$\mathbf{f}_{e}^{*}(\mathbf{z}_{e})=0$ throughout, and $\mathbf{f}_{\alpha}(\mathbf{x}^{*})=0$
for all $\alpha\in V\cup E$, the first line of \eqref{eq:fenchel-duality-gap}
can be simplified to be the $F(\mathbf{z})$ in \eqref{eq:def-dual-F}. 

We make another set of assumptions on Algorithm \ref{alg:Ext-Dyk}
that will allow us to prove our linear convergence result. 
\begin{assumption}
\label{assu:alg-assu}For Algorithm \ref{alg:Ext-Dyk}, we assume
that:
\begin{enumerate}
\item For all $\alpha\in V\cup E$ and $n\geq1$, there is a $w_{n,\alpha}$
such that $\alpha\in S_{n,w_{n,\alpha}}$. 
\item \label{enu:S-eq-V}$S_{n,1}=V$.
\end{enumerate}
\end{assumption}

Out plan is to prove the main result in Section \ref{sec:Main-result}
with Assumption \ref{assu:alg-assu}\eqref{enu:S-eq-V} first, then
remove it in Section \ref{sec:Lift-assu}.

\section{\label{sec:Main-result}Main result}

In this section, we state and prove the main theorem on linear convergence
of the distributed Dykstra's algorithm. Our proof is split into three
cases. For the first two cases, the proof in this section does not
rely on Assumption \ref{assu:alg-assu}\eqref{enu:S-eq-V}. For the
third case, we first prove our result by first assuming Assumption
\ref{assu:alg-assu}\eqref{enu:S-eq-V}. We then show how to lift
this assumption in Section \ref{sec:Lift-assu}.
\begin{thm}
\label{thm:the-thm}(Linear convergence of dual value) Suppose Assumptions
\ref{assu:the-assu} and \ref{assu:alg-assu} hold. For Algorithm
\ref{alg:Ext-Dyk}, there is a constant $r\in(0,1)$ such that $F(\mathbf{z}^{n+1,0})\leq rF(\mathbf{z}^{n-1,0})$.
Together with \eqref{eq:fenchel-duality-gap}, this implies that the
distance $\{\|x_{i}^{n,0}-x^{*}\|\}_{n\geq1}$ converges linearly
to zero for all $i\in V$. 
\end{thm}

We need positive parameters $\bar{\epsilon}$, $\theta_{D}$ and $\theta_{Z}$
to be small enough so that they satisfy $\bar{\epsilon}|V|(2\kappa_{2}M_{\max}+1)\leq\frac{1}{4}$,
$c_{2}(\theta_{Z},\theta_{D},0)>0$ and \eqref{eq_m:3-cond}, where
$c(\cdot)$ and $c_{2}(\cdot)$ are defined in \eqref{eq:d-T-x} and
\eqref{eq:last-case}, and the other constants are described in Assumption
\ref{assu:the-assu} and in the course of the proof. It is easy to
see that the parameters $\bar{\epsilon}$, $\theta_{D}$ and $\theta_{Z}$
can be chosen to satisfy these conditions. 

The first two cases of the proof of Theorem \ref{thm:the-thm} are
easier than the third case. To simplify notation, we let $\mathbf{z}_{\alpha}^{n,w}$
to be written simply as $\mathbf{z}_{\alpha}^{w}$ for all $w\in\{0,1,\dots,\bar{w}\}$
and $\alpha\in V\cup E$, and the dropping of ``$n$'' appears in
all other variables as well. Let $x_{i}^{w}$ be the $i$-th coordinate
of $\mathbf{x}^{w}$, and let $z_{i}^{w}$ be $[\mathbf{z}_{i}^{w}]_{i}$,
the $i$-th coordinate of $\mathbf{z}_{i}^{w}$. If $S_{n,k}=\{i\}$
for some $i\in V$, then $x_{i}^{k}$ and $z_{i}^{k}$ are the solutions
to the primal dual pair of subproblems\begin{subequations} 
\begin{align}
 & \begin{array}{c}
\underset{x}{\overset{\phantom{x}}{\min}}\frac{1}{2}\|(x_{i}^{k-1}+z_{i}^{k-1})-x\|^{2}+\delta_{C_{i}}(x)\end{array}\label{eq:PD-primal}\\
\text{ and } & \begin{array}{c}
\underset{z}{\overset{\phantom{z}}{\max}}\frac{1}{2}\|x_{i}^{k-1}+z_{i}^{k-1}\|^{2}-\frac{1}{2}\|(x_{i}^{k-1}+z_{i}^{k-1})-z\|^{2}-\delta_{C_{i}}^{*}(z).\end{array}\label{eq:PD-dual}
\end{align}
\end{subequations}By adjusting \eqref{eq:PD-dual}, we can easily
check that $z_{i}^{k}$ is the minimizer of 
\begin{equation}
\begin{array}{c}
\underset{z}{\overset{\phantom{x}}{\min}}\frac{1}{2}\|(x_{i}^{k-1}+z_{i}^{k-1})-z\|^{2}+\delta_{C_{i}}^{*}(z).\end{array}\label{eq:dual-min-form}
\end{equation}

\begin{proof}
[Proof of cases 1 and 2 of Theorem \ref{thm:the-thm}]The proof is
split into 3 cases:

\textbf{Case 1: $\|\mathbf{x}^{0}\|^{2}\leq\bar{\epsilon}\sum_{i\in V}\delta_{C_{i}}^{*}(z_{i}^{0})$}

Let $i^{*}$ be $\arg\max_{i\in V}\{\delta_{C_{i}}^{*}(z_{i}^{0})\}$.
We have 
\begin{equation}
\begin{array}{c}
\bar{\epsilon}|V|\delta_{C_{i^{*}}}^{*}(z_{i^{*}}^{0})\geq\bar{\epsilon}\underset{i\in V}{\overset{\phantom{i\in V}}{\sum}}\delta_{C_{i}}^{*}(z_{i}^{0})\overset{\scriptsize{\text{Case 1}}}{\geq}\|\mathbf{x}^{0}\|^{2}.\end{array}\label{eq:index-large-delta-star}
\end{equation}
Also, 
\begin{equation}
-\Big(1+\frac{\bar{\epsilon}}{2}\Big)\delta_{C_{i^{*}}}^{*}(z_{i^{*}}^{0})\overset{\scriptsize{\text{Case 1,\eqref{eq:index-large-delta-star}}}}{\leq}-\frac{1}{|V|}\Big(\frac{1}{2}\|\mathbf{x}^{0}\|^{2}+\sum_{i\in V}\delta_{C_{i}}^{*}(z_{i}^{0})\Big)\overset{\eqref{eq:def-dual-F}}{=}-\frac{1}{|V|}F(\mathbf{z}^{0}).\label{eq:case-1-all}
\end{equation}
We can assume that at index $k$, we have $S_{n,k}=\{i^{*}\}$ and
$i^{*}\notin S_{n,k'}$ for all $k'<k$. We have 2 cases. 

\textbf{Case 1a: $\|x_{i^{*}}^{k}\|^{2}\leq\frac{1}{2\kappa_{2}M_{\max}+1}\delta_{C_{i^{*}}}^{*}(z_{i^{*}}^{0})$,}

In this case, 
\begin{eqnarray}
 &  & \begin{array}{c}
\delta_{C_{i^{*}}}^{*}(z_{i^{*}}^{k})+\frac{1}{2}\|x_{i^{*}}^{k}\|^{2}\overset{\eqref{eq:delta-star-has-squared-values},z_{i^{*}}^{k}\in N_{C_{i}}(x_{i^{*}}^{k})}{\underset{\phantom{\eqref{eq:norm-z-i-bdd}}}{\leq}}\left(\kappa_{2}\|z_{i^{*}}^{k}\|+\frac{1}{2}\right)\|x_{i^{*}}^{k}\|^{2}\end{array}\nonumber \\
 & \overset{\eqref{eq:norm-z-i-bdd}}{\underset{\phantom{\eqref{eq:norm-z-i-bdd}}}{\leq}} & \begin{array}{c}
\left(\kappa_{2}M_{\max}+\frac{1}{2}\right)\|x_{i^{*}}^{k}\|^{2}\overset{\scriptsize{\text{Case 1a}}}{\underset{\phantom{\eqref{eq:norm-z-i-bdd}}}{\leq}}\frac{1}{2}\delta_{C_{i^{*}}}^{*}(z_{i^{*}}^{0})\leq\frac{1}{2}\delta_{C_{i^{*}}}^{*}(z_{i^{*}}^{0})+\frac{1}{2}\|x_{i^{*}}^{k-1}\|^{2}.\end{array}\label{eq:case-1a-1}
\end{eqnarray}
 Recall that $z_{i^{*}}^{0}=z_{i^{*}}^{k-1}$. Since $S_{n,k}=\{i^{*}\}$,
we also have $z_{i}^{k}=z_{i}^{k-1}$ for all $i\neq i^{*}$ and $x_{i}^{k}=x_{i}^{k-1}$
for all $i\neq i^{*}$. We have 
\begin{eqnarray}
 &  & \begin{array}{c}
F(\mathbf{z}^{k})\overset{\eqref{eq:def-dual-F}}{=}\underset{i\in V}{\overset{\phantom{i\in V}}{\sum}}\left(\delta_{C_{i}}^{*}(z_{i}^{k})+\frac{1}{2}\|x_{i}^{k}\|^{2}\right)\end{array}\label{eq:lin-rate-case-1a}\\
 & \leq & \begin{array}{c}
\underset{i\neq i^{*}}{\overset{\phantom{i\in V}}{\sum}}\left(\delta_{C_{i}}^{*}(z_{i}^{k})+\frac{1}{2}\|x_{i}^{k}\|^{2}\right)+\frac{1}{2}\delta_{C_{i^{*}}}^{*}(z_{i^{*}}^{0})+\frac{1}{2}\|x_{i^{*}}^{k-1}\|^{2}\end{array}\nonumber \\
 & \overset{\eqref{eq:def-dual-F},\eqref{eq:case-1a-1}}{\leq} & \begin{array}{c}
F(\mathbf{z}^{k-1})-\frac{1}{2}\delta_{C_{i^{*}}}^{*}(z_{i^{*}}^{0})\overset{\eqref{eq:case-1-all},F(\mathbf{z}^{k-1})\leq F(\mathbf{z}^{0})}{\underset{\phantom{\eqref{eq:case-1-all}}}{\leq}}\left(1-\frac{1}{|V|(2+\bar{\epsilon})}\right)F(\mathbf{z}^{0}).\end{array}\nonumber 
\end{eqnarray}

\textbf{Case 1b: $\|x_{i^{*}}^{k}\|^{2}\geq\frac{1}{2\kappa_{2}M_{\max}+1}\delta_{C_{i^{*}}}^{*}(z_{i^{*}}^{0})$}

Note that $\frac{1}{2}\|x_{i^{*}}^{k}-x_{i^{*}}^{0}\|^{2}$ is an
estimate of the decrease of the dual objective value. We choose $\bar{\epsilon}>0$
so that $\bar{\epsilon}|V|(2\kappa_{2}M_{\max}+1)\leq\frac{1}{4}$.
We have 
\begin{equation}
\begin{array}{c}
\|x_{i^{*}}^{0}\|^{2}\overset{\eqref{eq:index-large-delta-star}}{\leq}\bar{\epsilon}|V|\delta_{C_{i^{*}}}^{*}(z_{i^{*}}^{0})\overset{\scriptsize{\text{Case 1b}}}{\underset{\phantom{\scriptsize{\text{Case 1b}}}}{\leq}}\bar{\epsilon}|V|(2\kappa_{2}M_{\max}+1)\|x_{i^{*}}^{k}\|^{2}\leq\frac{1}{4}\|x_{i^{*}}^{k}\|^{2}.\end{array}\label{eq:case-1b-0}
\end{equation}
We then have 
\begin{equation}
\begin{array}{c}
\|x_{i^{*}}^{k}-x_{i^{*}}^{0}\|^{2}\geq(\|x_{i^{*}}^{k}\|-\|x_{i^{*}}^{0}\|)^{2}\overset{\eqref{eq:case-1b-0}}{\underset{\phantom{\eqref{eq:case-1b-0}}}{\geq}}\frac{1}{4}\|x_{i^{*}}^{k}\|^{2}\overset{\scriptsize{\text{Case 1b}}}{\underset{\phantom{\scriptsize{\text{Case 1b}}}}{\geq}}\frac{1}{4(2\kappa_{2}M_{\max}+1)}\delta_{C_{i^{*}}}^{*}(z_{i^{*}}^{0}).\end{array}\label{eq:case-1b-1}
\end{equation}
Then 
\begin{equation}
\begin{array}{c}
\underset{k'=1}{\overset{k}{\sum}}\|x_{i^{*}}^{k'}-x_{i^{*}}^{k'-1}\|^{2}\geq\frac{1}{k}\bigg(\underset{k'=1}{\overset{k}{\sum}}\|x_{i^{*}}^{k'}-x_{i^{*}}^{k'-1}\|\bigg)^{2}\geq\frac{1}{\bar{w}}\|x_{i^{*}}^{0}-x_{i^{*}}^{k}\|^{2}.\end{array}\label{eq:intermediate-ineq}
\end{equation}
We then have
\begin{eqnarray}
F(\mathbf{z}^{k}) & \overset{\eqref{eq:def-dual-F}}{=} & \begin{array}{c}
\underset{i\in V}{\overset{\phantom{i\in V}}{\sum}}\delta_{C_{i}}^{*}(z_{i}^{k})+\frac{1}{2}\|\mathbf{x}^{k}\|^{2}\end{array}\nonumber \\
 & \overset{\eqref{eq:Dykstra-min-subpblm}}{\underset{\phantom{\eqref{eq:Dykstra-min-subpblm}}}{\leq}} & \begin{array}{c}
\underset{i\in V}{\overset{\phantom{i\in V}}{\sum}}\delta_{C_{i}}^{*}(z_{i}^{0})+\frac{1}{2}\|\mathbf{x}^{0}\|^{2}-\underset{k'=1}{\overset{k}{\sum}}\frac{1}{2}\|x_{i^{*}}^{k'}-x_{i^{*}}^{k'-1}\|^{2}\end{array}\nonumber \\
 & \overset{\eqref{eq:def-dual-F},\eqref{eq:case-1b-1},\eqref{eq:intermediate-ineq}}{\leq} & \begin{array}{c}
F(\mathbf{z}^{0})-\frac{1}{8\bar{w}(2\kappa_{2}M_{\max}+1)}\delta_{C_{i^{*}}}^{*}(z_{i^{*}}^{0})\end{array}\nonumber \\
 & \overset{\eqref{eq:case-1-all}}{\leq} & \begin{array}{c}
\left(1-\frac{1}{4\bar{w}(2\kappa_{2}M_{\max}+1)(2+\bar{\epsilon})|V|}\right)F(\mathbf{z}^{0}).\end{array}\label{eq:lin-rate-case-1b}
\end{eqnarray}

\textbf{Case 2: $\|\mathbf{x}^{0}\|^{2}\geq\bar{\epsilon}\sum_{i\in V}\delta_{C_{i}}^{*}(z_{i}^{0})$,
and $\|P_{D^{\perp}}\mathbf{x}^{0}\|^{2}\geq\theta_{D}\|\mathbf{x}^{0}\|^{2}$. }

In this case, note that $d(\mathbf{x}^{0},D)=\|P_{D^{\perp}}\mathbf{x}^{0}\|\overset{\scriptsize{\text{Case 2}}}{\geq}\sqrt{\theta_{D}}\|\mathbf{x}^{0}\|$.
Since $D\overset{\eqref{eq:def-D-hyperplane}}{=}\cap_{e\in E}H_{e}$
and $H_{e}$ are hyperplanes, there is some $\kappa_{D}>0$ such that
$\max_{e\in E}d(\mathbf{x}^{0},H_{e})\geq\frac{1}{\kappa_{D}}d(\mathbf{x}^{0},D)$.
Let $e^{*}$ be such that $d(\mathbf{x}^{0},H_{e^{*}})=\max_{e\in E}d(\mathbf{x}^{0},H_{e})$,
and let $k$ be such that $\mathbf{x}^{k}\in H_{e^{*}}$, which exists
by Assumption \ref{assu:alg-assu}(1). We then have 
\begin{equation}
\begin{array}{c}
\|\mathbf{x}^{0}-\mathbf{x}^{k}\|\overset{\mathbf{x}^{k}\in H_{e^{*}}}{\geq}d(\mathbf{x}^{0},H_{e^{*}})\geq\frac{1}{\kappa_{D}}d(\mathbf{x}^{0},D)\overset{\scriptsize{\text{Case 2}}}{\underset{\phantom{\scriptsize{\text{Case 2}}}}{\geq}}\frac{1}{\kappa_{D}}\sqrt{\theta_{D}}\|\mathbf{x}^{0}\|.\end{array}\label{eq:case-2-ineq-chain}
\end{equation}
Now, 
\begin{equation}
\begin{array}{c}
\left(\frac{1}{2}+\frac{1}{\bar{\epsilon}}\right)\|\mathbf{x}^{0}\|^{2}\overset{\scriptsize{\text{Case 2}}}{\underset{\phantom{\scriptsize{\text{Cas}}}}{\geq}}\frac{1}{2}\|\mathbf{x}^{0}\|^{2}+\underset{i\in V}{\overset{\phantom{i\in V}}{\sum}}\delta_{C_{i}}^{*}(z_{i}^{0})\overset{\eqref{eq:def-dual-F}}{=}F(\mathbf{z}^{0}).\end{array}\label{eq:case-2-a}
\end{equation}
We have $\frac{1}{2}\sum_{i=1}^{k}\|\mathbf{x}^{i}-\mathbf{x}^{i-1}\|^{2}\geq\frac{1}{2w}\left(\sum_{i=1}^{k}\|\mathbf{x}^{i}-\mathbf{x}^{i-1}\|\right)^{2}\geq\frac{1}{2w}\|\mathbf{x}^{0}-\mathbf{x}^{k}\|^{2}$,
so 
\begin{eqnarray}
F(\mathbf{z}^{k}) & \overset{\eqref{eq:Dykstra-min-subpblm}}{\leq} & \begin{array}{c}
F(\mathbf{z}^{0})-\frac{1}{2}\underset{i=1}{\overset{k}{\sum}}\|\mathbf{x}^{i}-\mathbf{x}^{i-1}\|^{2}\leq F(\mathbf{z}^{0})-\frac{1}{2w}\|\mathbf{x}^{0}-\mathbf{x}^{k}\|^{2}\end{array}\label{eq:case-2-end}\\
 & \overset{\eqref{eq:case-2-ineq-chain}}{\underset{\phantom{\scriptsize{\text{Cas}}}}{\leq}} & \begin{array}{c}
F(\mathbf{z}^{0})-\frac{1}{2w\kappa_{D}^{2}}\theta_{D}\|\mathbf{x}^{0}\|^{2}\overset{\eqref{eq:case-2-a}}{\underset{\phantom{\scriptsize{\text{Cas}}}}{\leq}}\left(1-\frac{1}{w\kappa_{D}^{2}}\theta_{D}\frac{\bar{\epsilon}}{\bar{\epsilon}+2}\right)F(\mathbf{z}^{0}).\end{array}\nonumber 
\end{eqnarray}
Hence we are done.
\end{proof}
This leaves us with Case 3, i.e., 

\textbf{Case 3: $\|\mathbf{x}^{0}\|^{2}\geq\bar{\epsilon}\sum_{i\in V}\delta_{C_{i}}^{*}(z_{i}^{0})$,
and $\|P_{D^{\perp}}\mathbf{x}^{0}\|^{2}\leq\theta_{D}\|\mathbf{x}^{0}\|^{2}$. }

By the definition of $D$ in \eqref{eq:def-D-hyperplane}, all $|V|$
components of $P_{D}\mathbf{x}^{0}$ are equal to some value, which
we call $a$. Then we have the inequalities 
\begin{align}
 & \begin{array}{c}
\|P_{D}\mathbf{x}^{0}\|^{2}=\|\mathbf{x}^{0}\|^{2}-\|P_{D^{\perp}}\mathbf{x}^{0}\|^{2}\overset{\scriptsize{\text{Case 3}}}{\geq}(1-\theta_{D})\|\mathbf{x}^{0}\|^{2},\end{array}\label{eq:case-3-0}\\
 & \begin{array}{c}
\underset{i\in V}{\overset{\phantom{i\in V}}{\sum}}\|x_{i}^{0}-a\|^{2}=\|P_{D^{\perp}}\mathbf{x}^{0}\|^{2}\overset{\scriptsize{\text{Case 3}}}{\leq}\theta_{D}\|\mathbf{x}^{0}\|^{2}\overset{\eqref{eq:case-3-0}}{\leq}\frac{\theta_{D}}{1-\theta_{D}}\|P_{D}\mathbf{x}^{0}\|^{2}=\frac{\theta_{D}|V|}{1-\theta_{D}}\|a\|^{2},\end{array}\label{eq:case-3-1}\\
 & \begin{array}{c}
|V|\|a\|^{2}=\|P_{D}\mathbf{x}^{0}\|^{2}\leq\|\mathbf{x}^{0}\|.\end{array}\label{eq:case-3-2}
\end{align}
We have \begin{subequations}\label{eq_m:x-compare-a}
\begin{align}
 & \begin{array}{c}
\|x_{i}^{0}\|\leq\phantom{\big|}\|a\|+\|x_{i}^{0}-a\|\phantom{\big|}\overset{\eqref{eq:case-3-1}}{\underset{\phantom{\eqref{eq:case-3-1}}}{\leq}}\left(1+\sqrt{\frac{\theta_{D}}{1-\theta_{D}}|V|}\right)\|a\|.\end{array}\label{eq:x-leq-a}\\
\text{and } & \begin{array}{c}
\|x_{i}^{0}\|\geq\big|\|a\|-\|x_{i}^{0}-a\|\big|\overset{\eqref{eq:case-3-1}}{\underset{\phantom{\eqref{eq:case-3-1}}}{\geq}}\left(1-\sqrt{\frac{\theta_{D}}{1-\theta_{D}}|V|}\right)\|a\|.\end{array}\label{eq:x-geq-a}
\end{align}
\end{subequations} and 
\begin{equation}
\begin{array}{c}
\|a-x_{i}^{0}\|\overset{\eqref{eq:case-3-1}}{\leq}\sqrt{\frac{\theta_{D}}{1-\theta_{D}}|V|}\|a\|\overset{\eqref{eq:x-geq-a}}{\underset{\phantom{\eqref{eq:x-geq-a}}}{\leq}}\left(1-\sqrt{\frac{\theta_{D}}{1-\theta_{D}}|V|}\right)^{-1}\sqrt{\frac{\theta_{D}}{1-\theta_{D}}|V|}\|x_{i}^{0}\|.\end{array}\label{eq:a-minus-x}
\end{equation}

We now show that there is a constant $\tilde{\kappa}_{3}>0$ such
that $\|\hat{z}_{i}^{0}-\hat{z}_{i}^{r}\|\leq\tilde{\kappa}_{3}\|x_{i}^{0}\|$.
We have $\|\hat{z}_{i}^{0}-\hat{z}_{i}^{r}\|\overset{\eqref{eq:normals-linear-dec}}{\leq}\kappa_{3}\|\tilde{x}_{i}\|$,
where $\tilde{x}_{i}:=x_{i}^{n-1,p(n-1,i)}$, and $p(n-1,i)$ is the
index such that $i\notin S_{n-1,k}$ for all $k$ such that $p(n-1,i)<k\leq\bar{w}$.
If we have $\|\hat{z}_{i}^{0}-\hat{z}_{i}^{r}\|>\tilde{\kappa}_{3}\|x_{i}^{0}\|$,
then 
\begin{eqnarray*}
 &  & \begin{array}{c}
F(\mathbf{z}^{n-1,p(n-1.i)})\geq\frac{1}{2}\|\tilde{x}_{i}\|^{2}\overset{\eqref{eq:normals-linear-dec}}{\geq}\frac{1}{2\kappa_{3}^{2}}\|\hat{z}_{i}^{0}-\hat{z}_{i}^{r}\|^{2}\geq\frac{\tilde{\kappa}_{3}^{2}}{2\kappa_{3}^{2}}\|x_{i}^{0}\|^{2}\end{array}\\
 & \overset{\eqref{eq:x-geq-a}}{\geq} & \begin{array}{c}
\frac{\tilde{\kappa}_{3}^{2}}{2\kappa_{3}^{2}}\left(1-\sqrt{\frac{\theta_{D}|V|}{1-\theta_{D}}}\right)^{2}\|a\|^{2}\end{array}\\
 & \overset{\eqref{eq:case-3-1}}{\geq} & \begin{array}{c}
\frac{\tilde{\kappa}_{3}^{2}}{2\kappa_{3}^{2}}\left(1-\sqrt{\frac{\theta_{D}|V|}{1-\theta_{D}}}\right)^{2}\frac{1-\theta_{D}}{|V|}\|\mathbf{x}^{0}\|^{2}\overset{\eqref{eq:case-2-a}}{\geq}\frac{\tilde{\kappa}_{3}^{2}}{\kappa_{3}^{2}}\left(1-\sqrt{\frac{\theta_{D}|V|}{1-\theta_{D}}}\right)^{2}\frac{1-\theta_{D}}{|V|}\frac{\bar{\epsilon}}{2+\bar{\epsilon}}F(\mathbf{z}^{0}).\end{array}
\end{eqnarray*}
This would then give us $F(\mathbf{z}^{n-1,0})\geq\frac{\tilde{\kappa}_{3}^{2}}{\kappa_{3}^{2}}\left(1-\sqrt{\frac{\theta_{D}|V|}{1-\theta_{D}}}\right)^{2}\frac{1-\theta_{D}}{|V|}\frac{\bar{\epsilon}}{2+\bar{\epsilon}}F(\mathbf{z}^{n+1,0})$.
The parameter $\tilde{\kappa}_{3}$ can be chosen large enough so
that the coefficient of $F(\mathbf{z}^{n+1,0})$ is greater than $1$,
which once again leads to the conclusion in Theorem \ref{thm:the-thm}.
Therefore, we shall assume 
\begin{equation}
\|\hat{z}_{i}^{0}-\hat{z}_{i}^{r}\|\leq\tilde{\kappa}_{3}\|x_{i}^{0}\|\label{eq:bdd-by-tilde-kappa}
\end{equation}
 throughout. We now assume Assumption \ref{assu:alg-assu}\eqref{enu:S-eq-V},
and let $x_{i}^{+}$ and $z_{i}^{+}$ be $x_{i}^{1}$ and $z_{i}^{1}$
respectively. 
\begin{proof}
[Proof of case 3 of Theorem \ref{thm:the-thm}]We consider $\{z_{i}^{0}\}_{i\in V}$
and $\{z_{i}^{r}\}_{i\in V}$, where $z_{i}^{r}=P_{N_{C_{i}}(x^{*})}(z_{i}^{0})$.
Recall $M\subset[\mathbb{R}^{m}]^{|V|}$ defined as the set of optimal
multipliers defined in Assumption \ref{assu:the-assu}\eqref{enu:lin-reg}.
Let $(z_{1}^{p},\dots,z_{|V|}^{p})\in[\mathbb{R}^{m}]^{|V|}$ be 
\begin{equation}
(z_{1}^{p},\dots,z_{|V|}^{p})=P_{M}\big((z_{1}^{r},\dots,z_{|V|}^{r})\big),\label{eq:z-p-def-proj}
\end{equation}
where $(z_{1}^{r},\dots,z_{|V|}^{r})\in[\mathbb{R}^{m}]^{|V|}$. Let
$Z$ be $\sspan(\{z_{i}^{p}\}_{i\in V})$. Let $d$ be the direction
$P_{Z^{\perp}}a$. There are two subcases to consider. 

\textbf{Case 3a:} $\|P_{Z}a\|^{2}\leq\theta_{Z}\|a\|^{2}$ 

Since $d=P_{Z^{\perp}}a\in Z^{\perp}$, we have 
\begin{equation}
d\perp\hat{z}_{i}^{p}\text{ for all }i\in V.\label{eq:d-perp-z}
\end{equation}
We would be projecting $x_{i}^{0}+z_{i}^{0}$ onto $C_{i}$ for all
$i\in V$. Let an outer approximate of $C_{i}$ be 
\begin{equation}
P_{i}:=\{x:(\hat{z}_{i}^{0})^{T}x\leq\epsilon_{i},(z_{i}^{p})^{T}x\leq0\},\text{ where }\epsilon_{i}:=\delta_{C_{i}}^{*}(\hat{z}_{i}^{0}).\label{eq:poly-def-P-i}
\end{equation}
Since $C_{i}\subset P_{i}$, we have $\delta_{P_{i}}(\cdot)\leq\delta_{C_{i}}(\cdot)$,
and so $\delta_{P_{i}}^{*}(\cdot)\geq\delta_{C_{i}}^{*}(\cdot)$.
By the design of $P_{i}$, we have $\delta_{P_{i}}^{*}(z_{i}^{0})=\delta_{C_{i}}^{*}(z_{i}^{0})$.
Since $d\in Z^{\perp}$ and $\bar{x}\in Z$, we have $d^{T}\bar{x}=0$.
Proposition \ref{prop:sparsity}(2) implies that $\sum_{i\in V}\sum_{\alpha\in E}[\mathbf{z}_{\alpha}^{0}]_{i}=0$.
So we have 
\begin{equation}
\begin{array}{c}
\hat{d}^{T}\underset{i\in V}{\overset{\phantom{i\in V}}{\sum}}(x_{i}^{0}+z_{i}^{0})\overset{\scriptsize{\text{Prop \ref{prop:sparsity}}(2)}}{\underset{\phantom{\scriptsize{\text{Prop \ref{prop:sparsity}}}}}{=}}\hat{d}^{T}\underset{i\in V}{\overset{\phantom{i\in V}}{\sum}}(x_{i}^{0}+[\mathbf{z}_{i}^{0}]_{i}+\underset{\alpha\in E}{\overset{\phantom{i\in V}}{\sum}}[\mathbf{z}_{\alpha}^{0}]_{i})\overset{\eqref{eq:def-dual-F}}{=}\hat{d}^{T}\underset{i\in V}{\overset{\phantom{i\in V}}{\sum}}\bar{x}=0.\end{array}\label{eq:d-x-z-eq-0}
\end{equation}
Hence there is some $i$ such that 
\begin{equation}
\hat{d}^{T}(x_{i}^{0}+z_{i}^{0})\overset{\eqref{eq:d-x-z-eq-0}}{\leq}0.\label{eq:d-x-z-leq-0}
\end{equation}
Then we move ahead with this $i$ (without labeling it as $i^{*}$
to save notation). 

Since $d=P_{Z^{\perp}}a$, we have $d^{T}a=a^{T}P_{Z^{\perp}}a=a^{T}P_{Z^{\perp}}P_{Z^{\perp}}a=\|P_{Z^{\perp}}a\|^{2}$.
Note that \begin{subequations}\label{eq_m:d-and-d-minus-a}
\begin{align}
 & \|d\|^{2}=\|P_{Z^{\perp}}a\|^{2}=\|a\|^{2}-\|P_{Z}a\|^{2}\overset{\scriptsize{\text{Case 3a}}}{\geq}(1-\theta_{Z})\|a\|^{2}\label{eq:d-geq-a-norms}\\
\text{and } & \|d-a\|^{2}=\|P_{Z}a\|^{2}\overset{\scriptsize{\text{Case 3a}}}{\leq}\theta_{Z}\|a\|^{2},\label{eq:d-minus-a-leq-a-norms}
\end{align}
\end{subequations} so 
\begin{eqnarray}
d^{T}x_{i}^{0} & = & d^{T}a+d^{T}(x_{i}^{0}-a)\geq\|P_{Z^{\perp}}a\|^{2}-\|d\|\|x_{i}^{0}-a\|\label{eq:d-T-x}\\
 & \overset{\eqref{eq:d-geq-a-norms},\eqref{eq:case-3-1}}{\underset{\phantom{\eqref{eq:case-3-1}}}{\geq}} & \|d\|\left(\sqrt{1-\theta_{Z}}-\sqrt{\frac{\theta_{D}|V|}{1-\theta_{D}}}\right)\|a\|.\nonumber \\
 & \overset{\eqref{eq:x-leq-a}}{\geq} & \|d\|\underbrace{\left(\sqrt{1-\theta_{Z}}-\sqrt{\frac{\theta_{D}}{1-\theta_{D}}|V|}\right)\left(1+\sqrt{\frac{\theta_{D}}{1-\theta_{D}}|V|}\right)^{-1}}_{c(\theta_{Z},\theta_{D})}\|x_{i}^{0}\|.\nonumber 
\end{eqnarray}
Let $c(\theta_{Z},\theta_{D})$ be the formula marked above. Let $\hat{d}=d/\|d\|$.
We have 
\begin{equation}
\begin{array}{c}
\hat{d}^{T}\hat{z}_{i}^{0}=\frac{1}{\|z_{i}^{0}\|}\left(\hat{d}^{T}(x_{i}^{0}+z_{i}^{0})-\hat{d}^{T}x_{i}^{0}\right)\overset{\eqref{eq:norm-z-i-bdd},\eqref{eq:d-x-z-leq-0},\eqref{eq:d-T-x}}{\underset{\phantom{\eqref{eq:d-x-z-leq-0}}}{\leq}}\frac{-c(\theta_{Z},\theta_{D})}{M_{\max}}\|x_{i}^{0}\|.\end{array}\label{eq:d-z-hat}
\end{equation}

We then project $x_{i}^{0}+z_{i}^{0}$ onto $P_{i}$. Suppose $\hat{z}_{i}^{0}$
is close enough to $\hat{z}_{i}^{p}$ so that $\|\hat{z}_{i}^{p}-\hat{z}_{i}^{0}\|\leq\frac{1}{2}$.
Then 
\begin{equation}
\begin{array}{c}
(\hat{z}_{i}^{p})^{T}z_{i}^{0}=(\hat{z}_{i}^{0})^{T}z_{i}^{0}+(\hat{z}_{i}^{p}-\hat{z}_{i}^{0})^{T}z_{i}^{0}\geq\|z_{i}^{0}\|(1-\|\hat{z}_{i}^{p}-\hat{z}_{i}^{0}\|)\overset{\eqref{eq:norm-z-i-bdd}}{\underset{\phantom{\eqref{eq:norm-z-i-bdd}}}{\geq}}\frac{1}{2}M_{\min}.\end{array}\label{eq:z-p-z-geq-half-M}
\end{equation}
If we assume that $x_{i}^{0}$ is close enough to $x^{*}$ so that
$\|x_{i}^{0}\|\leq\frac{1}{3}M_{\min}$, then $(\hat{z}_{i}^{p})^{T}x_{i}^{0}\geq-\|\hat{z}_{i}^{p}\|\|x_{i}^{0}\|=-\|x_{i}^{0}\|\geq-\frac{1}{3}M_{\min}$,
and so 
\[
\begin{array}{c}
(\hat{z}_{i}^{p})^{T}(x_{i}^{0}+z_{i}^{0})\overset{\eqref{eq:z-p-z-geq-half-M}}{\underset{\phantom{\eqref{eq:z-p-z-geq-half-M}}}{\geq}}-\frac{1}{3}M_{\min}+\frac{1}{2}M_{\min}=\frac{1}{6}M_{min}>0.\end{array}
\]
This means that $x_{i}^{0}+z_{i}^{0}$ does not satisfy the second
inequality in the definition of $P_{i}$ in \eqref{eq:poly-def-P-i},
so at least one of the inequalities there must be active at $P_{P_{i}}(x_{i}^{0}+z_{i}^{0})$.
We let the point $P_{P_{i}}(x_{i}^{0}+z_{i}^{0})$ be $\tilde{x}_{i}^{+}$.
\begin{claim*}
Recall that $\lim_{(\theta_{Z},\theta_{D})\to(0,0)}c(\theta_{Z},\theta_{D})\overset{\eqref{eq:d-T-x}}{=}1$.
Let $\kappa_{5}$ be $\frac{3\kappa_{4}|V|(M_{\max}\tilde{\kappa}_{3}+1)}{M_{\min}}+\tilde{\kappa}_{3}$,
which is checked to be greater than $1$. Suppose $\theta_{Z},\theta_{D}>0$
are chosen small enough so that the following conditions hold: \begin{subequations}\label{eq_m:3-cond}
\begin{align}
 & \begin{array}{c}
\frac{2\kappa_{4}|V|(M_{\max}\tilde{\kappa}_{3}+1)}{M_{\min}}\frac{1+\sqrt{\frac{\theta_{D}}{1-\theta_{D}}|V|}}{1-\sqrt{\frac{\theta_{D}}{1-\theta_{D}}|V|}}+\tilde{\kappa}_{3}\leq\kappa_{5},\end{array}\label{eq:3-cond-1}\\
 & \begin{array}{c}
\left(1-\sqrt{\frac{\theta_{D}|V|}{1-\theta_{D}}}\right)^{-1}\left(\sqrt{\frac{\theta_{D}|V|}{1-\theta_{D}}}+\theta_{Z}\right)-\left(1+\sqrt{\frac{\theta_{D}|V|}{1-\theta_{D}}}\right)^{-1}\frac{\sqrt{1-\theta_{Z}}c(\theta_{Z},\theta_{D})}{M_{\max}\kappa_{5}}\leq\frac{-1}{2M_{\max}\kappa_{5}},\end{array}\label{eq:3-cond-2}\\
 & \begin{array}{c}
c(\theta_{Z},\theta_{D})\geq\frac{1}{2M_{\max}\kappa_{5}}.\end{array}\label{eq:3-cond-3}
\end{align}
\end{subequations}Then $\|x_{i}^{0}-\tilde{x}_{i}^{+}\|\geq\frac{1}{2M_{\max}\kappa_{5}}\|x_{i}^{0}\|$. 
\end{claim*}
We now prove the claim. For $\tilde{x}_{i}^{+}=P_{P_{i}}(x_{i}^{0}+z_{i}^{0})$,
there are three different cases.

\textbf{Case 3a-1:} Only the constraint $(\hat{z}_{i}^{p})^{T}x\leq0$
in \eqref{eq:poly-def-P-i} is active at $\tilde{x}_{i}^{+}$.

If that active constraint is $(\hat{z}_{i}^{p})^{T}x\leq0$, then
by the KKT conditions, $\tilde{x}_{i}^{+}$ would be of the form $\tilde{x}_{i}^{+}=x_{i}^{0}+z_{i}^{0}-\lambda\hat{z}_{i}^{p}$,
and hence 
\begin{equation}
d^{T}\tilde{x}_{i}^{+}=d^{T}(x_{i}^{0}+z_{i}^{0})-\lambda d^{T}\hat{z}_{i}^{p}\overset{\eqref{eq:d-perp-z},\eqref{eq:d-x-z-leq-0}}{\leq}0.\label{eq:dx-leq-0}
\end{equation}
Then 
\[
\begin{array}{c}
\|x_{i}^{0}-\tilde{x}_{i}^{+}\|\geq\hat{d}^{T}(x_{i}^{0}-\tilde{x}_{i}^{+})\overset{\eqref{eq:d-T-x},\eqref{eq:dx-leq-0}}{\geq}c(\theta_{Z},\theta_{D})\|x_{i}^{0}\|\overset{\eqref{eq:3-cond-3}}{\geq}\frac{1}{2M_{\max}\kappa_{5}}\|x_{i}^{0}\|.\end{array}
\]

\textbf{Case 3a-2: }Both constraints in \eqref{eq:poly-def-P-i} are
active at $\tilde{x}_{i}^{+}$. 

\textbf{Step 1:} Bounding $\|\hat{z}_{i}^{0}-\hat{z}_{i}^{p}\|$. 

For all $i\in V$, we have 
\begin{align}
\|z_{i}^{p}-z_{i}^{r}\| & \leq\|(z_{1}^{p},\dots,z_{|V|}^{p})-(z_{1}^{r},\dots,z_{|V|}^{r})\|\label{eq:z-p-minus-r}\\
 & \overset{\eqref{eq:z-p-def-proj}}{=}d\big((z_{1}^{r},\dots,z_{|V|}^{r}),M\big)\overset{\scriptsize{\text{Assu \ref{assu:the-assu}\eqref{enu:lin-reg}}}}{\leq}\kappa_{4}d\big((z_{1}^{r},\dots,z_{|V|}^{r}),M_{2}\big).\nonumber 
\end{align}
The projection of $(z_{1}^{r},\dots,z_{|V|}^{r})$ onto $M_{2}$ is
$(z_{1}^{r}-\delta,z_{2}^{r}-\delta,\dots,z_{|V|}^{r}-\delta)$, where
$M_{2}$ is as defined in \eqref{eq:set-M2} and $\delta=\frac{1}{|V|}(\sum_{i\in V}z_{i}^{r}-|V|\bar{x})$.
This means that 
\begin{equation}
\begin{array}{c}
d\big((z_{1}^{r},\dots,z_{|V|}^{r}),M_{2}\big)=\sqrt{|V|}\|\delta\|=\frac{1}{\sqrt{|V|}}\left\Vert \underset{i\in V}{\overset{\phantom{i\in V}}{\sum}}z_{i}^{r}-|V|\bar{x}\right\Vert .\end{array}\label{eq:d-to-M-2}
\end{equation}
For the parameters $(z_{1}^{r},\dots,z_{|V|}^{r})$, we note from
Proposition \ref{prop:sparsity} that $z_{i}^{0}=[\mathbf{z}_{i}^{0}]_{i}$,
$[\mathbf{z}_{i}^{0}]_{j}=0$ for all $j\neq i$ and $\sum_{j\in V}[\sum_{\alpha\in E}\mathbf{z}_{\alpha}^{0}]_{j}=0$,
which gives 
\begin{eqnarray}
\begin{array}{c}
\underset{i\in V}{\overset{\phantom{i\in V}}{\sum}}z_{i}^{r}-|V|\bar{x}\end{array} & \overset{\scriptsize{\text{Prop \ref{prop:sparsity}}}}{=} & \begin{array}{c}
\underset{i\in V}{\overset{\phantom{i\in V}}{\sum}}(z_{i}^{r}-z_{i}^{0})+\underset{j\in V}{\overset{\phantom{i\in V}}{\sum}}\left[\underset{\alpha\in E}{\overset{\phantom{i\in V}}{\sum}}\mathbf{z}_{\alpha}^{0}+\underset{i\in V}{\overset{\phantom{i\in V}}{\sum}}\mathbf{z}_{i}^{0}-\bar{\mathbf{x}}\right]_{j}\end{array}\nonumber \\
 & = & \begin{array}{c}
\underset{i\in V}{\overset{\phantom{i\in V}}{\sum}}(z_{i}^{r}-z_{i}^{0})-\underset{i\in V}{\overset{\phantom{i\in V}}{\sum}}x_{i}^{0}.\end{array}\label{eq:step-1-identity}
\end{eqnarray}
Recall that $z_{i}^{r}=P_{N_{C_{i}}(x^{*})}(z_{i}^{0})$, and Assumption
\ref{assu:the-assu}\eqref{enu:osc-of-norms-hat}. This gives $\|\hat{z}_{i}^{r}-\hat{z}_{i}^{0}\|\overset{\eqref{eq:bdd-by-tilde-kappa}}{\leq}\tilde{\kappa}_{3}\|x_{i}^{0}\|$
and
\begin{equation}
\|z_{i}^{r}-z_{i}^{0}\|=d(z_{i}^{0},N_{C_{i}}(x^{*}))\leq\big\| z_{i}^{0}-\hat{z}_{i}^{r}\|z_{i}^{0}\|\big\|=\|z_{i}^{0}\|\|\hat{z}_{i}^{r}-\hat{z}_{i}^{0}\|\overset{\eqref{eq:norm-z-i-bdd},\eqref{eq:bdd-by-tilde-kappa}}{\leq}M_{\max}\tilde{\kappa}_{3}\|x_{i}^{0}\|.\label{eq:step-1-bdd-1}
\end{equation}
So 
\begin{eqnarray}
\begin{array}{c}
\end{array} &  & \begin{array}{c}
\left\Vert \underset{i\in V}{\overset{\phantom{i\in V}}{\sum}}z_{i}^{r}-|V|\bar{x}\right\Vert \overset{\eqref{eq:step-1-identity}}{\leq}\underset{i\in V}{\overset{\phantom{i\in V}}{\sum}}\|z_{i}^{r}-z_{i}^{0}\|+\underset{i\in V}{\overset{\phantom{i\in V}}{\sum}}\|x_{i}^{0}\|\end{array}\label{eq:delta-term-bdd}\\
 & \overset{\eqref{eq:step-1-bdd-1}}{\leq} & \begin{array}{c}
\underset{i\in V}{\overset{\phantom{i\in V}}{\sum}}(M_{\max}\tilde{\kappa}_{3}+1)\|x_{i}^{0}\|\leq|V|(M_{\max}\tilde{\kappa}_{3}+1)\|\mathbf{x}^{0}\|.\end{array}\nonumber 
\end{eqnarray}
Hence, for all $i\in V$, we have 
\begin{equation}
\|z_{i}^{p}-z_{i}^{r}\|\overset{\eqref{eq:z-p-minus-r},\eqref{eq:d-to-M-2},\eqref{eq:delta-term-bdd}}{\leq}\kappa_{4}\sqrt{|V|}(M_{\max}\tilde{\kappa}_{3}+1)\|\mathbf{x}^{0}\|.\label{eq:z-p-and-r}
\end{equation}
Also, 
\begin{eqnarray}
 &  & \begin{array}{c}
\|\hat{z}_{i}^{p}-\hat{z}_{i}^{r}\|\leq\left\Vert \frac{z_{i}^{p}}{\|z_{i}^{p}\|}-\frac{z_{i}^{r}}{\|z_{i}^{p}\|}\right\Vert +\left\Vert \frac{z_{i}^{r}}{\|z_{i}^{p}\|}-\frac{z_{i}^{r}}{\|z_{i}^{r}\|}\right\Vert \end{array}\label{eq:z-hat-p-and-r}\\
 & = & \begin{array}{c}
\frac{1}{\|z_{i}^{p}\|}\|z_{i}^{p}-z_{i}^{r}\|+\|z_{i}^{r}\|\frac{|\|z_{i}^{p}\|-\|z_{i}^{r}\||}{\|z_{i}^{p}\|\|z_{i}^{r}\|}\end{array}\nonumber \\
 & \leq & \begin{array}{c}
\frac{1}{\|z_{i}^{p}\|}\|z_{i}^{p}-z_{i}^{r}\|+\frac{1}{\|z_{i}^{p}\|}\|z_{i}^{p}-z_{i}^{r}\|\overset{\eqref{eq:z-p-and-r},\eqref{eq:norm-z-i-bdd}}{\leq}\frac{2\kappa_{4}\sqrt{|V|}(M_{\max}\tilde{\kappa}_{3}+1)}{M_{\min}}\|\mathbf{x}^{0}\|.\end{array}\nonumber 
\end{eqnarray}
Since $z_{i}^{r}=P_{N_{C_{i}}(x^{*})}(z_{i}^{0})$, Assumption \ref{assu:the-assu}\eqref{enu:osc-of-norms-hat}
shows us that $\|\hat{z}_{i}^{0}-\hat{z}_{i}^{r}\|\leq\kappa_{3}\|x_{i}^{0}\|$.
Note that for any $i\in V$, 
\begin{equation}
\begin{array}{c}
\|x_{i}^{0}\|^{2}\leq\|\mathbf{x}^{0}\|^{2}\overset{\eqref{eq:x-leq-a}}{\leq}|V|\left(1+\sqrt{\frac{\theta_{D}|V|}{1-\theta_{D}}}\right)^{2}\|a\|^{2}\overset{\eqref{eq:x-geq-a}}{\underset{\phantom{\eqref{eq:x-geq-a}}}{\leq}}|V|\left(\frac{1+\sqrt{\frac{\theta_{D}|V|}{1-\theta_{D}}}}{1-\sqrt{\frac{\theta_{D}|V|}{1-\theta_{D}}}}\right)^{2}\|x_{i}^{0}\|^{2}.\end{array}\label{eq:equiv-x-i-mathbf-x}
\end{equation}
We thus have 
\begin{equation}
\|\hat{z}_{i}^{0}-\hat{z}_{i}^{p}\|\leq\|\hat{z}_{i}^{0}-\hat{z}_{i}^{r}\|+\|\hat{z}_{i}^{r}-\hat{z}_{i}^{p}\|\overset{\eqref{eq:bdd-by-tilde-kappa},\eqref{eq:z-hat-p-and-r},\eqref{eq:equiv-x-i-mathbf-x},\eqref{eq:3-cond-1}}{\leq}\kappa_{5}\|x_{i}^{0}\|.\label{eq:neq-z0-zp}
\end{equation}

\textbf{Step 2:} Showing $\|x_{i}^{0}-\tilde{x}_{i}^{+}\|$ is large
enough. 

Since both constraints in $P_{i}$ (see \eqref{eq:poly-def-P-i})
are tight at $\tilde{x}_{i}^{+}$, the projection of $x_{i}^{0}+z_{i}^{0}$
onto $P_{i}$ is equivalent to the projection of $P_{(\hat{z}_{i}^{p})^{\perp}}(x_{i}^{0}+z_{i}^{0})$
onto $\{x:(P_{(\hat{z}_{i}^{p})^{\perp}}(\hat{z}_{i}^{0}))^{T}x=\epsilon_{i}\}$.
We have 
\begin{equation}
\|P_{(\hat{z}_{i}^{p})^{\perp}}(\hat{z}_{i}^{0})\|=\|P_{(\hat{z}_{i}^{p})^{\perp}}(\hat{z}_{i}^{0}-\hat{z}_{i}^{p})\|\leq\|\hat{z}_{i}^{0}-\hat{z}_{i}^{p}\|\overset{\eqref{eq:neq-z0-zp}}{\leq}\kappa_{5}\|x_{i}^{0}-x^{*}\|.\label{eq:normal-z-length}
\end{equation}
Note that by the KKT conditions, $P_{(\hat{z}_{i}^{p})^{\perp}}(\hat{z}_{i}^{0})=\hat{z}_{i}^{0}-\lambda\hat{z}_{i}^{p}$
for some $\lambda\in\mathbb{R}$. So 
\begin{equation}
\begin{array}{c}
\hat{d}^{T}(P_{(\hat{z}_{i}^{p})^{\perp}}(\hat{z}_{i}^{0}))=\hat{d}^{T}(\hat{z}_{i}^{0}-\lambda\hat{z}_{i}^{p})\overset{\eqref{eq:d-perp-z}}{\underset{\phantom{\eqref{eq:d-perp-z}}}{=}}\hat{d}^{T}\hat{z}_{i}^{0}\overset{\eqref{eq:d-z-hat}}{\leq}\frac{-c(\theta_{Z},\theta_{D})}{M_{\max}}\|x_{i}^{0}-x^{*}\|<0.\end{array}\label{eq:normal-z-tilt}
\end{equation}
Then we have 
\begin{equation}
\begin{array}{c}
\frac{d^{T}(P_{(\hat{z}_{i}^{p})^{\perp}}(\hat{z}_{i}^{0}))}{\|d\|\|P_{(\hat{z}_{i}^{p})^{\perp}}(\hat{z}_{i}^{0})\|}\overset{\scriptsize{\text{\eqref{eq:normal-z-length},\eqref{eq:normal-z-tilt}, terms}<0}}{\leq}\frac{-c(\theta_{Z},\theta_{D})}{M_{\max}\kappa_{5}}.\end{array}\label{eq:inner-pdts-very-neg}
\end{equation}

Also,
\begin{eqnarray}
 &  & \begin{array}{c}
\frac{(x_{i}^{0})^{T}P_{(\hat{z}_{i}^{p})^{\perp}}(\hat{z}_{i}^{0})}{\|x_{i}^{0}\|\|P_{(\hat{z}_{i}^{p})^{\perp}}(\hat{z}_{i}^{0})\|}=\frac{(x_{i}^{0}-d)^{T}P_{(\hat{z}_{i}^{p})^{\perp}}(\hat{z}_{i}^{0})}{\|x_{i}^{0}\|\|P_{(\hat{z}_{i}^{p})^{\perp}}(\hat{z}_{i}^{0})\|}+\frac{d^{T}P_{(\hat{z}_{i}^{p})^{\perp}}(\hat{z}_{i}^{0})}{\|x_{i}^{0}\|\|P_{(\hat{z}_{i}^{p})^{\perp}}(\hat{z}_{i}^{0})\|}\end{array}\label{eq:hyperplane-slope}\\
 & = & \begin{array}{c}
\frac{((x_{i}^{0}-a)+(a-d))^{T}P_{(\hat{z}_{i}^{p})^{\perp}}(\hat{z}_{i}^{0})}{\|x_{i}^{0}\|\|P_{(\hat{z}_{i}^{p})^{\perp}}(\hat{z}_{i}^{0})\|}+\frac{\|d\|}{\|x_{i}^{0}\|}\frac{d^{T}P_{(\hat{z}_{i}^{p})^{\perp}}(\hat{z}_{i}^{0})}{\|d\|\|P_{(\hat{z}_{i}^{p})^{\perp}}(\hat{z}_{i}^{0})\|}\end{array}\nonumber \\
 & \overset{\eqref{eq:inner-pdts-very-neg}}{\leq} & \begin{array}{c}
\frac{\|x_{i}^{0}-a\|}{\|x_{i}^{0}\|}+\frac{\|a-d\|}{\|x_{i}^{0}\|}-\frac{c(\theta_{Z},\theta_{D})}{M_{\max}\kappa_{5}}\frac{\|d\|}{\|a\|}\frac{\|a\|}{\|x_{i}^{0}\|}\end{array}\nonumber \\
 & \overset{\eqref{eq:a-minus-x},\eqref{eq_m:d-and-d-minus-a}}{\leq} & \begin{array}{c}
\left(1-\sqrt{\frac{\theta_{D}|V|}{1-\theta_{D}}}\right)^{-1}\left(\sqrt{\frac{\theta_{D}|V|}{1-\theta_{D}}}\right)+\theta_{Z}\frac{\|a\|}{\|x_{i}^{0}\|}-\frac{c(\theta_{Z},\theta_{D})}{M_{\max}\kappa_{5}}\frac{\sqrt{1-\theta_{Z}}\|a\|}{\|x_{i}^{0}\|}.\end{array}\nonumber \\
 & \overset{\eqref{eq_m:x-compare-a},\eqref{eq:3-cond-2}}{\leq} & \begin{array}{c}
-\frac{1}{2M_{\max}\kappa_{5}}.\end{array}\nonumber 
\end{eqnarray}
Note that $\tilde{x}_{i}^{+}$ is the deflection of $x_{i}^{0}$
along the normal $P_{(\hat{z}_{i}^{p})^{\perp}}(\hat{z}_{i}^{0})$,
i.e., $\tilde{x}_{i}^{+}=x_{i}^{0}+\lambda P_{(\hat{z}_{i}^{p})^{\perp}}(\hat{z}_{i}^{0})$
for some $\lambda\geq0$. Moreover, we have $\left(\frac{P_{(\hat{z}_{i}^{p})^{\perp}}(\hat{z}_{i}^{0})}{\|P_{(\hat{z}_{i}^{p})^{\perp}}(\hat{z}_{i}^{0})\|}\right)^{T}\tilde{x}_{i}^{+}=\frac{\epsilon_{i}}{\|P_{(\hat{z}_{i}^{p})^{\perp}}(\hat{z}_{i}^{0})\|}$
since the two constraints in the definition of $P_{i}$ in \eqref{eq:poly-def-P-i}
are tight. The distance of $\tilde{x}_{i}^{+}$ must be at least 
\begin{eqnarray*}
\begin{array}{c}
\|\tilde{x}_{i}^{+}-x_{i}^{0}\|\end{array} & \geq & \begin{array}{c}
\left(\frac{P_{(\hat{z}_{i}^{p})^{\perp}}(\hat{z}_{i}^{0})}{\|P_{(\hat{z}_{i}^{p})^{\perp}}(\hat{z}_{i}^{0})\|}\right)^{T}(\tilde{x}_{i}^{+}-x_{i}^{0})\end{array}\\
 & \overset{\eqref{eq:hyperplane-slope}}{\geq} & \begin{array}{c}
\frac{\epsilon_{i}}{\|P_{(\hat{z}_{i}^{p})^{\perp}}(\hat{z}_{i}^{0})\|}-\left(-\frac{1}{2M_{\max}\kappa_{5}}\right)\|x_{i}^{0}\|\geq\frac{1}{2M_{\max}\kappa_{5}}\|x_{i}^{0}\|,\end{array}
\end{eqnarray*}
which concludes the proof for this case.

\textbf{Case 3a-3:} Only the constraint $(\hat{z}_{i}^{0})^{T}x\leq\epsilon_{i}$
in \eqref{eq:poly-def-P-i} is active at $\tilde{x}_{i}^{+}$. 

We now show that this case is impossible by showing that $(\hat{z}_{i}^{0})^{T}\tilde{x}_{i}^{+}=\epsilon_{i}>0$
and $(\hat{z}_{i}^{p})^{T}\tilde{x}_{i}^{+}\leq0$ cannot hold at
the same time. We have $\|x_{i}^{0}-a\|\overset{\eqref{eq:case-3-1}}{\leq}\sqrt{\frac{\theta_{D}|V|}{1-\theta_{D}}}\|a\|$.
By the nonexpansiveness of the projection operation, we have 
\begin{eqnarray}
 &  & \begin{array}{c}
\|P_{(\hat{z}_{i}^{p})^{\perp}}(x_{i}^{0})-d\|\overset{\eqref{eq:d-perp-z}}{\leq}\|x_{i}^{0}-d\|\leq\|x_{i}^{0}-a\|+\|d-a\|\end{array}\nonumber \\
 & \!\!\!\!\!\!\!\!\!\!\overset{\scriptsize{\text{\eqref{eq:case-3-1},\eqref{eq:d-minus-a-leq-a-norms}}}}{\underset{\phantom{\scriptsize{\text{\eqref{eq:d-minus-a-leq-a-norms}}}}}{\leq}} & \begin{array}{c}
\left(\sqrt{\frac{\theta_{D}|V|}{1-\theta_{D}}}+\sqrt{\theta_{Z}}\right)\|a\|\overset{\eqref{eq:d-geq-a-norms}}{\leq}\frac{1}{\sqrt{1-\theta_{Z}}}\left(\sqrt{\frac{\theta_{D}|V|}{1-\theta_{D}}}+\sqrt{\theta_{Z}}\right)\|d\|.\end{array}\label{eq:x-i-to-d}
\end{eqnarray}
Define $x_{i}'$ to be the point such that $x_{i}'=x_{i}^{0}+\lambda\hat{z}_{i}^{0}$
and $(\hat{z}_{i}^{p})^{T}x_{i}'=0$. Note that $P_{(\hat{z}_{i}^{p})^{\perp}}(x_{i}^{0})$
is of the form $x_{i}^{0}+\lambda\hat{z}_{i}^{p}$ with $(\hat{z}_{i}^{p})^{T}P_{(\hat{z}_{i}^{p})^{\perp}}(x_{i}^{0})=0$.
Further arithmetic gives us 
\begin{equation}
\begin{array}{c}
x_{i}'=x_{i}^{0}-\frac{(\hat{z}_{i}^{p})^{T}x_{i}^{0}}{(\hat{z}_{i}^{0})^{T}\hat{z}_{i}^{p}}\hat{z}_{i}^{0}\text{ and }P_{(\hat{z}_{i}^{p})^{\perp}}(x_{i}^{0})=x_{i}^{0}-[(\hat{z}_{i}^{p})^{T}x_{i}^{0}]\hat{z}_{i}^{p}.\end{array}\label{eq:x-prime-and-p}
\end{equation}
Now 
\begin{equation}
\begin{array}{c}
(\hat{z}_{i}^{p})^{T}x_{i}^{0}=(\hat{z}_{i}^{p})^{T}(x_{i}^{0}-d)+(\hat{z}_{i}^{p})^{T}d\leq\|x_{i}^{0}-d\|\overset{\eqref{eq:x-i-to-d}}{\leq}\frac{1}{\sqrt{1-\theta_{Z}}}\left(\sqrt{\frac{\theta_{D}|V|}{1-\theta_{D}}}+\sqrt{\theta_{Z}}\right)\|d\|.\end{array}\label{eq:zp-T-x}
\end{equation}
 Also, $\left\Vert \frac{1}{(\hat{z}_{i}^{0})^{T}\hat{z}_{i}^{p}}\hat{z}_{i}^{0}-\hat{z}_{i}^{p}\right\Vert \leq\|\hat{z}_{i}^{0}-\hat{z}_{i}^{p}\|+\left(\frac{1-(\hat{z}_{i}^{0})^{T}\hat{z}_{i}^{p}}{(\hat{z}_{i}^{0})^{T}\hat{z}_{i}^{p}}\right)$.
Since $\|\hat{z}_{i}^{0}-\hat{z}_{i}^{p}\|$ can be made arbitrarily
small by \eqref{eq:bdd-by-tilde-kappa} and $\left(\frac{1-(\hat{z}_{i}^{0})^{T}\hat{z}_{i}^{p}}{(\hat{z}_{i}^{0})^{T}\hat{z}_{i}^{p}}\right)=\left(\frac{\frac{1}{2}\|\hat{z}_{i}^{0}-\hat{z}_{i}^{p}\|^{2}}{\frac{1}{4}(\|\hat{z}_{i}^{0}+\hat{z}_{i}^{p}\|^{2}-\|\hat{z}_{i}^{0}-\hat{z}_{i}^{p}\|^{2})}\right)$,
we can assume that there is an $\gamma_{1}$ such that $\left\Vert \frac{1}{(\hat{z}_{i}^{0})^{T}\hat{z}_{i}^{p}}\hat{z}_{i}^{0}-\hat{z}_{i}^{p}\right\Vert \leq\gamma_{1}$
throughout. So
\begin{equation}
\begin{array}{c}
\|x_{i}'-P_{(\hat{z}_{i}^{p})^{\perp}}(x_{i}^{0})\|\overset{\eqref{eq:x-prime-and-p}}{=}[(\hat{z}_{i}^{p})^{T}x_{i}^{0}]\left\Vert \frac{1}{(\hat{z}_{i}^{0})^{T}\hat{z}_{i}^{p}}\hat{z}_{i}^{0}-\hat{z}_{i}^{p}\right\Vert \overset{\eqref{eq:zp-T-x}}{\leq}\frac{1}{\sqrt{1-\theta_{Z}}}\left(\sqrt{\frac{\theta_{D}|V|}{1-\theta_{D}}}+\sqrt{\theta_{Z}}\right)\gamma_{1}\|d\|.\end{array}\label{eq:x-minus-p}
\end{equation}

By the KKT conditions, the point $\tilde{x}_{i}^{+}$ has the form
$\tilde{x}_{i}^{+}=x_{i}^{0}+\lambda z_{i}^{0}$ for some $\lambda\geq0$.
We show that  points of the form $x_{i}^{0}+\lambda z_{i}^{0}$,
where $\lambda\in\mathbb{R}$, cannot satisfy both $(\hat{z}_{i}^{0})^{T}(x_{i}^{0}+\lambda z_{i}^{0})\geq0$
and $(\hat{z}_{i}^{p})^{T}(x_{i}^{0}+\lambda z_{i}^{0})\leq0$ at
the same time. Since $x_{i}^{0}-x_{i}'$ is a multiple of $z_{i}^{0}$,
we can prove our results for points of the form $x_{i}'+\lambda z_{i}^{0}$.
 Now, 
\begin{eqnarray*}
(\hat{z}_{i}^{0})^{T}(x_{i}') & = & \begin{array}{c}
(P_{\mathbb{R}(\hat{z}_{i}^{p})}\hat{z}_{i}^{0}+P_{(\hat{z}_{i}^{p})^{\perp}}\hat{z}_{i}^{0})^{T}x_{i}'\end{array}\\
 & = & \begin{array}{c}
(P_{\mathbb{R}(\hat{z}_{i}^{p})}\hat{z}_{i}^{0})^{T}x_{i}'+(P_{(\hat{z}_{i}^{p})^{\perp}}\hat{z}_{i}^{0})^{T}(x_{i}'-d)+(P_{(\hat{z}_{i}^{p})^{\perp}}\hat{z}_{i}^{0})^{T}(d)\end{array}\\
 & \overset{\eqref{eq:inner-pdts-very-neg}}{\leq} & \begin{array}{c}
0+\|P_{(\hat{z}_{i}^{p})^{\perp}}\hat{z}_{i}^{0}\|\|x_{i}'-d\|-\frac{c(\theta_{Z},\theta_{D})}{M_{\max}\kappa_{5}}\|P_{(\hat{z}_{i}^{p})^{\perp}}\hat{z}_{i}^{0}\|\|d\|.\end{array}
\end{eqnarray*}

In view of $\|x_{i}'-d\|\leq\|x_{i}'-P_{(\hat{z}_{i}^{p})^{\perp}}(x_{i}^{0})\|+\|P_{(\hat{z}_{i}^{p})^{\perp}}(x_{i}^{0})-d\|$,
\eqref{eq:x-i-to-d} and \eqref{eq:x-minus-p}, and the fact that
$\lim_{(\theta_{Z},\theta_{D})\to(0,0)}c(\theta_{Z},\theta_{D})=1$,
we can choose $\theta_{Z},\theta_{D}>0$ small enough so that $(\hat{z}_{i}^{0})^{T}(x_{i}')\leq0$.
So if $(\hat{z}_{i}^{0})^{T}(x_{i}'+\lambda z_{i}^{0})\geq0$, then
$\lambda\geq0$, which implies that $(\hat{z}_{i}^{p})^{T}(x_{i}'+\lambda z_{i}^{0})=\lambda(\hat{z}_{i}^{p})^{T}z_{i}^{0}>0$.
This completes the proof of the claim. $\hfill\triangle$

Let the minimizer of $\frac{1}{2}\|(x_{i}^{0}+z_{i}^{0})-\cdot\|^{2}+\delta_{P_{i}}^{*}(\cdot)$
be $\tilde{z}_{i}^{+}$. It is standard to obtain $\tilde{x}_{i}^{+}+\tilde{z}_{i}^{+}=x_{i}^{0}+z_{i}^{0}$.
We have
\begin{eqnarray}
 &  & \begin{array}{c}
\frac{1}{2}\|(x_{i}^{0}+z_{i}^{0})-z_{i}^{0}\|^{2}+\delta_{C_{i}}^{*}(z_{i}^{0})\overset{\eqref{eq:poly-def-P-i}}{=}\frac{1}{2}\|(x_{i}^{0}+z_{i}^{0})-z_{i}^{0}\|^{2}+\delta_{P_{i}}^{*}(z_{i}^{0})\end{array}\nonumber \\
 & \overset{\scriptsize{\tilde{z}_{i}^{+}\text{ minimizer}}}{\geq} & \begin{array}{c}
\frac{1}{2}\|(x_{i}^{0}+z_{i}^{0})-\tilde{z}_{i}^{+}\|^{2}+\delta_{P_{i}}^{*}(\tilde{z}_{i}^{+})+\frac{1}{2}\|z_{i}^{0}-\tilde{z}_{i}^{+}\|^{2}\end{array}\nonumber \\
 & \overset{\delta_{P_{i}}^{*}(\cdot)\geq\delta_{C_{i}}^{*}(\cdot)}{\geq} & \begin{array}{c}
\frac{1}{2}\|(x_{i}^{0}+z_{i}^{0})-\tilde{z}_{i}^{+}\|^{2}+\delta_{C_{i}}^{*}(\tilde{z}_{i}^{+})+\frac{1}{2}\|x_{i}^{0}-\tilde{x}_{i}^{+}\|^{2}\end{array}\nonumber \\
 & \overset{\scriptsize{z_{i}^{+}\text{ minimizer}}}{\geq} & \begin{array}{c}
\frac{1}{2}\|(x_{i}^{0}+z_{i}^{0})-z_{i}^{+}\|^{2}+\delta_{C_{i}}^{*}(z_{i}^{+})+\frac{1}{2}\|x_{i}^{0}-\tilde{x}_{i}^{+}\|^{2}.\end{array}\label{eq:3a-finale}
\end{eqnarray}
Note that $\|x_{i}^{0}\|\overset{\eqref{eq:x-geq-a}}{\geq}\left(1-\sqrt{\frac{\theta_{D}|V|}{1-\theta_{D}}}\right)\|a\|\overset{\eqref{eq:case-3-1}}{\geq}\left(1-\sqrt{\frac{\theta_{D}|V|}{1-\theta_{D}}}\right)\sqrt{\frac{1-\theta_{D}}{|V|}}\|\mathbf{x}^{0}\|$.
Also, $\frac{\bar{\epsilon}+2}{2\bar{\epsilon}}\|\mathbf{x}^{0}\|^{2}\overset{\eqref{eq:case-2-a}}{\geq}F(\mathbf{z}^{0})$.
Therefore 
\begin{eqnarray}
 &  & \begin{array}{c}
F(\mathbf{z}^{1})\overset{\eqref{eq:def-dual-F}}{=}\underset{i\in V}{\overset{\phantom{i\in V}}{\sum}}\big(\delta_{C_{i}}^{*}(z_{i}^{+})+\frac{1}{2}\|x_{i}^{+}\|^{2}\big)\end{array}\label{eq:case-3a-end}\\
 & \overset{\eqref{eq:3a-finale}}{\underset{\phantom{\eqref{eq:3a-finale}}}{\leq}} & \begin{array}{c}
\underset{i\in V}{\overset{\phantom{i\in V}}{\sum}}\big(\delta_{C_{i}}^{*}(z_{i}^{0})+\frac{1}{2}\|x_{i}^{0}\|^{2}\big)-\frac{1}{2}\|x_{i}^{0}-\tilde{x}_{i}^{+}\|^{2}\overset{\scriptsize{\text{\eqref{eq:def-dual-F}}}}{=}F(\mathbf{z}^{0})-\frac{1}{2}\|x_{i}^{0}-\tilde{x}_{i}^{+}\|^{2}\end{array}\nonumber \\
 & \overset{\scriptsize{\text{Claim}}}{\leq} & \begin{array}{c}
F(\mathbf{z}^{0})-\frac{\phantom{M_{\max}}1\phantom{M_{\max}}}{4(M_{\max}\kappa_{5})^{2}}\left(1-\sqrt{\frac{\theta_{D}}{1-\theta_{D}}|V|}\right)^{2}\frac{1-\theta_{D}}{|V|}\frac{\bar{\epsilon}}{2+\bar{\epsilon}}F(\mathbf{z}^{0}).\end{array}\nonumber 
\end{eqnarray}
This once again leads to linear convergence. 

\textbf{Case 3b:} $\|P_{Z}a\|^{2}\geq\theta_{Z}\|a\|^{2}$. 

For each $i\in V$, define the hyperplanes $H_{i}$, $H_{i,0}$ and
$H_{i,0}^{p}$ by 
\[
H_{i}:=\{x:(\hat{z}_{i}^{+})^{T}x=\delta_{C_{i}}^{*}(\hat{z}_{i}^{+})\},\,H_{i,0}:=\{x:(\hat{z}_{i}^{+})^{T}x=0\}\text{ and }H_{i,0}^{p}:=\{x:(\hat{z}_{i}^{p})^{T}x=0\}.
\]
Recall that by Assumption \ref{assu:the-assu}(4), $z_{i}^{0}$ are
big enough so that $x_{i}^{0}+z_{i}^{0}$ is always outside $C_{i}$,
so that $P_{C_{i}}(x_{i}^{0}+z_{i}^{0})$ is onto the boundary of
$C_{i}$ (and not in the interior of $C_{i}$). Recall that the dual
vectors after the projection are $\{z_{i}^{+}\}_{i\in V}$. The term
$\delta_{C_{i}}^{*}(\hat{z}_{i}^{0})$ in the definition of $H_{i}$
implies that $H_{i}$ is a supporting hyperplane of $C_{i}$ at $x_{i}^{+}$
with normal vector $\hat{z}_{i}^{+}$. Due to the fact that the dual
function is decreasing, we have $\delta_{C_{i'}}^{*}(z_{i'}^{+})+\frac{1}{2}\|x_{i'}^{+}\|^{2}\leq\delta_{C_{i'}}^{*}(z_{i'}^{0})+\frac{1}{2}\|x_{i'}^{0}\|^{2}$
for all $i'\in V$, so 
\begin{equation}
\begin{array}{c}
\|x_{i'}^{+}\|^{2}\leq2\delta_{C_{i'}}^{*}(z_{i'}^{+})+\|x_{i'}^{+}\|^{2}\leq2\underset{j\in V}{\overset{\phantom{i\in V}}{\sum}}\delta_{C_{j}}^{*}(z_{j}^{0})+\|\mathbf{x}^{0}\|^{2}\overset{\scriptsize{\text{Case 3}}}{\underset{\phantom{\scriptsize{\text{Case 3}}}}{\leq}}(1+\frac{2}{\bar{\epsilon}})\|\mathbf{x}^{0}\|^{2}.\end{array}\label{eq:for-o-term}
\end{equation}
If a point $x_{i'}^{+}$ is on $C_{i'}$, then the distance of the
supporting hyperplane of $C_{i'}$ at $x_{i'}^{+}$ to the origin
is $o(\|x_{i'}^{+}\|)\overset{\eqref{eq:for-o-term}}{=}o(\|\mathbf{x}^{0}\|)$
by Assumption \ref{assu:the-assu}(3). (We actually have $O(\|\mathbf{x}^{0}\|^{2})$,
but $o(\|\mathbf{x}^{0}\|)$ is enough for this part of the proof.)
So we have $d(0,H_{i})=o(\|x_{i}^{0}\|)$. Since $\|x_{i}^{0}\|\overset{\eqref{eq_m:x-compare-a}}{\in}\Theta(\|a\|)$,
the term $\delta_{C_{i}}^{*}(\hat{z}_{i}^{+})$ is $o(\|a\|)$, for
any $\tilde{\epsilon}_{1}>0$, we have $\tilde{\epsilon}_{1}\|a\|\geq d(0,H_{i})$
for all $i\in V$ if $\mathbf{x}^{0}$ is close enough to $\mathbf{x}^{*}$,
which gives 
\begin{equation}
d(a,H_{i})\geq d(a,H_{i,0})-\tilde{\epsilon}_{1}\|a\|\text{ for all }i\in V.\label{eq:deflect-off-zero}
\end{equation}
We have 
\begin{equation}
\begin{array}{c}
d(a,\cap_{i\in V}H_{i,0}^{p})=d(a,Z^{\perp})=\|P_{Z}a\|\overset{\scriptsize{\text{Case 3b}}}{\underset{\phantom{\scriptsize{\text{C}}}}{\geq}}\sqrt{\theta_{Z}}\|a\|.\end{array}\label{eq:dist-to-norm-a}
\end{equation}
We have $|(\hat{z}_{i}^{p})^{T}a|\leq|(\hat{z}_{i}^{+})^{T}a|+\|\hat{z}_{i}^{p}-\hat{z}_{i}^{+}\|\|a\|$.
Also, $d(a,H_{i,0})=|(\hat{z}_{i}^{+})^{T}a|$ and $|(\hat{z}_{i}^{p})^{T}a|=d(a,H_{i,0}^{p})$,
which leads us to $d(a,H_{i,0})\geq d(a,H_{i,0}^{p})-\|\hat{z}_{i}^{p}-\hat{z}_{i}^{+}\|\|a\|$.
Recall $\|\hat{z}_{i}^{p}-\hat{z}_{i}^{0}\|$ can be arbitrarily small
by \eqref{eq:neq-z0-zp}. Note also that $\|z_{i}^{0}-z_{i}^{+}\|=\|x_{i}^{0}-x_{i}^{+}\|$,
and the latter can be arbitrarily small. Also, by Assumption \ref{assu:the-assu}(4),
$\|z_{i}^{0}\|,\|z_{i}^{+}\|\geq M_{\min}$ so $\|\hat{z}_{i}^{0}-\hat{z}_{i}^{+}\|$
can be arbitrarily small. Thus we can make $\|\hat{z}_{i}^{p}-\hat{z}_{i}^{+}\|\leq\tilde{\epsilon}_{1}$.
So 
\begin{equation}
d(a,H_{i,0})=|(\hat{z}_{i}^{+})^{T}a|\geq|(\hat{z}_{i}^{p})^{T}a|-\|\hat{z}_{i}^{p}-\hat{z}_{i}^{+}\|\|a\|\geq d(a,H_{i,0}^{p})-\tilde{\epsilon}_{1}\|a\|.\label{eq:off-normal-to-normal}
\end{equation}
Next, by Assumption \ref{assu:the-assu}(1), we have 
\begin{equation}
\begin{array}{c}
|(\hat{z}_{i}^{p})^{T}a|=d(a,H_{i,0}^{p})\overset{\scriptsize{\text{Assu \ref{assu:the-assu}(1)}}}{\underset{\phantom{\scriptsize{\text{C}}}}{\geq}}\frac{1}{\kappa_{1}}d(a,\cap_{j\in V}H_{j,0}^{p})\overset{\eqref{eq:dist-to-norm-a}}{\underset{\phantom{\scriptsize{\text{C}}}}{\geq}}\frac{\sqrt{\theta_{Z}}}{\kappa_{1}}\|a\|.\end{array}\label{eq:use-lin-reg}
\end{equation}
We have 
\begin{equation}
\begin{array}{c}
d(a,H_{i})\overset{\eqref{eq:deflect-off-zero},\eqref{eq:off-normal-to-normal}}{\underset{\phantom{\scriptsize{\text{C}}}}{\geq}}d(a,H_{i,0}^{p})-2\tilde{\epsilon}_{1}\|a\|\overset{\eqref{eq:use-lin-reg}}{\underset{\phantom{\scriptsize{\text{C}}}}{\geq}}\left(\frac{\sqrt{\theta_{Z}}}{\kappa_{1}}-2\tilde{\epsilon}_{1}\right)\|a\|.\end{array}\label{eq:last-dist-to-norm}
\end{equation}
We have 
\begin{eqnarray}
\|x_{i}^{0}-x_{i}^{+}\| & \overset{x_{i}^{+}\in H_{i}}{\geq} & d(x_{i}^{0},H_{i})\geq d(a,H_{i})-\|x_{i}^{0}-a\|\label{eq:last-case}\\
 & \overset{\eqref{eq:case-3-1},\eqref{eq:last-dist-to-norm}}{\geq} & \left(\frac{\sqrt{\theta_{Z}}}{\kappa_{1}}-2\tilde{\epsilon}_{1}-\sqrt{\frac{\theta_{D}|V|}{1-\theta_{D}}}\right)\|a\|\nonumber \\
 & \overset{\eqref{eq:case-3-1}}{\geq} & \underbrace{\left(\frac{\sqrt{\theta_{Z}}}{\kappa_{1}}-2\tilde{\epsilon}_{1}-\sqrt{\frac{\theta_{D}|V|}{1-\theta_{D}}}\right)\sqrt{\frac{1-\theta_{D}}{|V|}}}_{c_{2}(\theta_{Z},\theta_{D},\tilde{\epsilon}_{1})}\|\mathbf{x}^{0}\|.\nonumber 
\end{eqnarray}
Since $\theta_{Z},\theta_{D},\tilde{\epsilon}_{1}>0$ are chosen so
that $c_{2}(\theta_{Z},\theta_{D},\tilde{\epsilon}_{1})>0$, we have
\[
\begin{array}{c}
\|x_{i}^{0}-x_{i}^{+}\|^{2}\overset{\eqref{eq:last-case}}{\geq}c_{2}(\theta_{Z},\theta_{D},\tilde{\epsilon}_{1})^{2}\|\mathbf{x}^{0}\|^{2}\overset{\eqref{eq:case-2-a}}{\geq}(\frac{1}{2}+\bar{\epsilon})^{-1}c_{2}(\theta_{Z},\theta_{D},\tilde{\epsilon}_{1})^{2}F(\mathbf{z}^{0}).\end{array}
\]
This leads to linear convergence like in the last three lines of
\eqref{eq:case-3a-end}. 
\end{proof}

\section{\label{sec:Lift-assu}Lifting Assumption \ref{assu:alg-assu}\eqref{enu:S-eq-V}}

In this section, we show how to adjust the proof of the main result
in Section \ref{sec:Main-result} so that Assumption \ref{assu:alg-assu}\eqref{enu:S-eq-V}
can be lifted. We let $z_{i}^{+}$ and $x_{i}^{+}$ be what they were
in the proof of Theorem \ref{thm:the-thm} in Section \ref{sec:Main-result}.
We shall treat case 3a first, and then explain the similarities in
case 3b.

We can assume that there is an index $k$ such that $i\notin S_{n,k'}$
for all $k'\in\{1,\dots,k-1\}$ (which implies $z_{i}^{0}=z_{i}^{k-1}$)
and $S_{n,k}=\{i\}$. Let the operator $T:\mathbb{R}^{n}\to\mathbb{R}^{n}$
be $T(x')=\arg\min_{x}\frac{1}{2}\|x'-x\|^{2}+\delta_{P_{i}}(x)$.
Define $\tilde{x}_{i}^{k}$ as 
\[
\tilde{x}_{i}^{k}:=T(x_{i}^{k-1}+z_{i}^{k-1})=T(x_{i}^{k-1}+z_{i}^{0}).
\]
Note also that $\tilde{x}_{i}^{+}=T(x_{i}^{0}+z_{i}^{0})$. Since
$\partial\delta_{P_{i}}(\cdot)$ is a monotone operator, the operator
$T(\cdot)$ is nonexpansive (see for example the textbook \cite{BauschkeCombettes11}),
which gives $\|\tilde{x}_{i}^{k}-\tilde{x}_{i}^{+}\|\leq\|x_{i}^{k-1}-x_{i}^{0}\|$.
We have 
\begin{equation}
\|x_{i}^{0}-\tilde{x}_{i}^{+}\|\leq\|x_{i}^{0}-x_{i}^{k-1}\|+\|x_{i}^{k-1}-\tilde{x}_{i}^{k}\|+\|\tilde{x}_{i}^{k}-\tilde{x}_{i}^{+}\|\leq2\|x_{i}^{0}-x_{i}^{k-1}\|+\|x_{i}^{k-1}-\tilde{x}_{i}^{k}\|.\label{eq:use-nonexpansive}
\end{equation}
Then
\begin{align}
 & \begin{array}{c}
\underset{k'=1}{\overset{k-1}{\sum}}\|\mathbf{x}^{k'}-\mathbf{x}^{k'-1}\|^{2}+\|x_{i}^{k-1}-\tilde{x}_{i}^{k}\|^{2}\geq\underset{k'=1}{\overset{k-1}{\sum}}\|x_{i}^{k'}-x_{i}^{k'-1}\|^{2}+\|x_{i}^{k-1}-\tilde{x}_{i}^{k}\|^{2}\end{array}\nonumber \\
\geq & \begin{array}{c}
\frac{1}{k-1}\left(\underset{k'=1}{\overset{k-1}{\sum}}\|x_{i}^{k'}-x_{i}^{k'-1}\|\right)^{2}+\|x_{i}^{k-1}-\tilde{x}_{i}^{k}\|^{2}\end{array}\label{eq:bdd-of-squares}\\
\geq & \begin{array}{c}
\frac{1}{\bar{w}-1}(\|x_{i}^{k-1}-x_{i}^{0}\|^{2}+\|x_{i}^{k-1}-\tilde{x}_{i}^{k}\|^{2})\end{array}\nonumber \\
\geq & \begin{array}{c}
\frac{1}{2\bar{w}}(\|x_{i}^{k-1}-x_{i}^{0}\|+\|x_{i}^{k-1}-\tilde{x}_{i}^{k}\|)^{2}\overset{\eqref{eq:use-nonexpansive}}{\geq}\frac{1}{8\bar{w}}\|x_{i}^{0}-\tilde{x}_{i}^{+}\|^{2}.\end{array}\nonumber 
\end{align}
 The same steps as \eqref{eq:3a-finale} leads us to 
\begin{align}
 & \begin{array}{c}
\frac{1}{2}\|(x_{i}^{k-1}+z_{i}^{k-1})-z_{i}^{k-1}\|^{2}+\delta_{C_{i}}^{*}(z_{i}^{k-1})\end{array}\label{eq:new-dec-step}\\
\geq & \begin{array}{c}
\frac{1}{2}\|(x_{i}^{k-1}+z_{i}^{k-1})-z_{i}^{k}\|^{2}+\delta_{C_{i}}^{*}(z_{i}^{k})+\frac{1}{2}\|x_{i}^{k-1}-\tilde{x}_{i}^{k}\|^{2}.\end{array}\nonumber 
\end{align}
Once again, the steps similar to \eqref{eq:case-3a-end} gives 
\begin{eqnarray*}
F(\mathbf{z}^{k}) & \overset{\eqref{eq:new-dec-step}}{\leq} & \begin{array}{c}
F(\mathbf{z}^{k-1})-\frac{1}{2}\|x_{i}^{k-1}-\tilde{x}_{i}^{k}\|^{2}\end{array}\\
 & \overset{\eqref{eq:Dykstra-min-subpblm}}{\leq} & \begin{array}{c}
F(\mathbf{z}^{0})-\underset{k'=1}{\overset{k-1}{\sum}}\frac{1}{2}\|\mathbf{x}^{k'}-\mathbf{x}^{k'-1}\|^{2}-\frac{1}{2}\|x_{i}^{k-1}-\tilde{x}_{i}^{k}\|^{2}\end{array}\\
 & \overset{\eqref{eq:bdd-of-squares}}{\leq} & \begin{array}{c}
F(\mathbf{z}^{0})-\frac{1}{16\bar{w}}\|x_{i}^{0}-\tilde{x}_{i}^{+}\|^{2}\end{array}\\
 & \overset{\scriptsize{\text{Claim}}}{\leq} & \begin{array}{c}
\left(1-\frac{\phantom{M_{\max}}1\phantom{M_{\max}}}{32\bar{w}(M_{\max}\kappa_{5})^{2}}\left(1-\sqrt{\frac{\theta_{D}}{1-\theta_{D}}|V|}\right)^{2}\frac{1-\theta_{D}}{|V|}\frac{\bar{\epsilon}}{2+\bar{\epsilon}}\right)F(\mathbf{z}^{0}).\end{array}
\end{eqnarray*}
The adjustments for case 3b is similar, except that the set $P_{i}$
is set to be $C_{i}$, and $\tilde{x}_{i}^{+}$ and $\tilde{x}_{i}^{k}$
can be replaced by $x_{i}^{+}$ and $x_{i}^{k}$ respectively.  

\bibliographystyle{amsalpha}
\bibliography{../refs}

\end{document}